\documentclass[final,leqno,onefignum,onetabnum]{siamltex}

\usepackage{epsfig,epsf,fancybox}
\usepackage{amsmath}
\usepackage{mathrsfs}
\usepackage{amssymb}
\usepackage{amsfonts}
\usepackage{graphicx}
\usepackage{color}
\usepackage{stmaryrd}
\usepackage{multirow}
\usepackage{booktabs}
\usepackage{accents}
\usepackage{cite}
\usepackage{float}
\usepackage[ruled,vlined,linesnumbered]{algorithm2e}
\usepackage{mathtools}  
\usepackage{enumitem}
\usepackage{diagbox}

\usepackage{kantlipsum,setspace}

\usepackage{mathtools}

\usepackage[normalem]{ulem} 

\newcommand{\bm}[1]{\boldsymbol{#1}}

\newcommand{\va}{{\mathbf{a}}}
\newcommand{\vb}{{\mathbf{b}}}
\newcommand{\vc}{{\mathbf{c}}}
\newcommand{\vd}{{\mathbf{d}}}
\newcommand{\ve}{{\mathbf{e}}}

\newcommand{\vg}{{\mathbf{g}}}
\newcommand{\vh}{{\mathbf{h}}}

\newcommand{\vq}{{\mathbf{q}}}

\newcommand{\vs}{{\mathbf{s}}}

\newcommand{\vx}{{\mathbf{x}}}
\newcommand{\vy}{{\mathbf{y}}}
\newcommand{\vz}{{\mathbf{z}}}

\newcommand{\vA}{{\mathbf{A}}}
\newcommand{\vB}{{\mathbf{B}}}
\newcommand{\vC}{{\mathbf{C}}}
\newcommand{\vD}{{\mathbf{D}}}

\newcommand{\vH}{{\mathbf{H}}}
\newcommand{\vI}{{\mathbf{I}}}

\newcommand{\vQ}{{\mathbf{Q}}}

\newcommand{\cH}{{\mathcal{H}}}

\newcommand{\cL}{{\mathcal{L}}}

\newcommand{\cN}{{\mathcal{N}}}

\newcommand{\cX}{{\mathcal{X}}}

\newcommand{\vareps}{\varepsilon}

\newcommand{\EE}{\mathbb{E}} 
\newcommand{\RR}{\mathbb{R}} 
\newcommand{\vzero}{\mathbf{0}} 
\newcommand{\vone}{{\mathbf{1}}} 

\newcommand{\dist}{\mathrm{dist}}    
\newcommand{\Prob}{{\mathrm{Prob}}} 
\newcommand{\Proj}{{\mathrm{Proj}}} 

\newcommand{\st}{\mbox{ s.t. }}

\DeclareMathOperator*{\argmin}{arg\,min} 

\newcommand{\bc}{\begin{center}}
\newcommand{\ec}{\end{center}}

\newcommand{\bdm}{\begin{displaymath}}
\newcommand{\edm}{\end{displaymath}}

\newcommand{\beq}{\begin{equation}}
\newcommand{\eeq}{\end{equation}}

\newcommand{\bfl}{\begin{flushleft}}
\newcommand{\efl}{\end{flushleft}}

\newcommand{\bt}{\begin{tabbing}}
\newcommand{\et}{\end{tabbing}}

\newcommand{\beqn}{\begin{eqnarray}}
\newcommand{\eeqn}{\end{eqnarray}}

\newcommand{\beqs}{\begin{align*}} 
\newcommand{\eeqs}{\end{align*}}  

\newtheorem{remark}{Remark}[section]
\newtheorem{assumption}{Assumption}
\newtheorem{setting}{Setting}

\begin{document}

\title{Primal-dual stochastic gradient method for convex programs with many functional constraints\thanks{This work is partly supported by NSF grant DMS-1719549.}}

\author{Yangyang Xu\thanks{\url{xuy21@rpi.edu}. Department of Mathematical Sciences, Rensselaer Polytechnic Institute, Troy, New York.}
}

\date{\today}

\maketitle

\begin{abstract}
Stochastic gradient method (SGM) has been popularly applied to solve optimization problems with objective that is stochastic or an average of many functions. Most existing works on SGMs assume that the underlying problem is unconstrained or has an easy-to-project constraint set. In this paper, we consider problems that have a stochastic objective and also many functional constraints. For such problems, it could be extremely expensive to project a point to the feasible set, or even compute subgradient and/or function value of all constraint functions. To find solutions of these problems, we propose a novel { (adaptive)} SGM based on the classical augmented Lagrangian function. Within every iteration, it inquires a stochastic subgradient of the objective, and a subgradient and the function value of one randomly sampled constraint function. Hence, the per-iteration complexity is low. We establish its convergence rate for convex problems and also problems with strongly convex objective. It can achieve the optimal $O(1/\sqrt{k})$ convergence rate for convex case and nearly optimal $O\big((\log k)/k\big)$ rate for strongly convex case. Numerical experiments on a sample approximation problem of the robust portfolio selection and quadratically constrained quadratic programming are conducted to demonstrate its efficiency.

\vspace{0.3cm}

\vspace{0.3cm}


\end{abstract}

\section{Introduction}

In this paper, we consider the constrained stochastic program
\begin{equation}\label{eq:ccp}
\min_{\vx\in X} f_0(\vx)\equiv \EE_\xi [F_0(\vx;\xi)], \st \ f_j(\vx)\le 0, j=1,\ldots, M,
\end{equation}
where $X$ is a convex set in $\RR^n$, $\xi$ is a random variable, and $f_j$ is a convex function for each $j=0,1,\ldots,M$. All nonlinear optimization problems in $\RR^n$ can be formulated in the form of \eqref{eq:ccp}. We are particularly interested in the case that $M$ is a large number. 

To find a solution of \eqref{eq:ccp}, we aim at designing a novel primal-dual stochastic gradient method (SGM). We assume an oracle, which can return a stochastic approximation of a subgradient of $f_0$, and also the function value and a deterministic subgradient of each $f_j$ at any inquired point $\vx\in X$. Since $M$ is big, it would be computationally very expensive if at every update, we inquire the objective value and/or subgradient of all $f_j$'s. Based on this observation, our algorithm, at every iteration, will simply call the oracle to return subgradients and function values of a few sampled constraint functions. 

The algorithm is derived based on the classical augmented Lagrangian function (c.f. \cite{rockafellar1973multiplier, rockafellar1976augmented}) of an equivalent rescaled variant of \eqref{eq:ccp}, i.e., 
$$\cL_\beta(\vx,\vz)= f_0(\vx) + \Psi_\beta(\vx, \vz).$$
Here, $\beta>0$ is the penalty parameter, $\vz$ is the Lagrangian multiplier or dual variable,
\begin{equation}\label{eq:def-Psi}
\Psi_\beta(\vx, \vz) = \frac{1}{M}\sum_{j=1}^M \psi_\beta\big(f_j(\vx), z_j\big),
\end{equation}
and
\begin{equation}\label{eq:psi}
\psi_\beta(u,v)=\left\{\begin{array}{ll} uv+\frac{\beta}{2}u^2,& \text{ if }\beta u+v \ge 0,\\[0.15cm]
-\frac{v^2}{2\beta}, & \text{ if }\beta u+v < 0.
\end{array}\right.
\end{equation}
Note that $\Psi_\beta$ is convex in $\vx$ and concave in $\vz$. Given $\beta>0$, the augmented dual function is defined as 
\begin{equation}\label{eq:dual-fun}
d_\beta(\vz)=\min_{\vx\in X}\cL_\beta(\vx,\vz).
\end{equation}

At each iteration $k$, we first sample one constraint function $f_{j_k}$. Secondly we call the oracle to obtain a stochastic subgradient $\vg_0^k$ of $f_0$, and a subgradient $\tilde\nabla f_{j_k}(\vx^k)$ and the function value of $f_{j_k}$ at $\vx^k$. Let  
\begin{equation}\label{eq:subgrad-psi}
\vh^k=[\beta  f_{j_k}(\vx^k)+z_{j_k}^k]_+\tilde\nabla f_{j_k}(\vx^k).
\end{equation}
Then $\vg^k_0+\vh^k$ is a stochastic subgradient of $\cL_{\beta }$ with respect to $\vx$.
Thirdly we perform a projected stochastic subgradient update as in \eqref{eq:update-x} to the primal variable $\vx$, and finally we update dual variable $z_{j_k}$. 

\begin{algorithm}[h]
\caption{Primal-dual stochastic gradient (PDSG) method for \eqref{eq:ccp}}\label{alg:pd-sg}
\textbf{Initialization:} choose $\vx^1\in X$, $\vz^1=\vzero$, and $\beta>0$\;
\For{$k=1,2,\ldots$}{
Pick $j_k\in [M]$ uniformly at random\;
Call the oracle to return a stochastic subgradient $\vg_0^k$ of $f_0$ and subgradient and function value of $f_{j_k}$ at $\vx^k$\;  
Obtain $\vh^k$ in \eqref{eq:subgrad-psi}, choose $\vD_k\succ 0$, and update the primal variable $\vx$ by 
\begin{equation}\label{eq:update-x}
{ \vx^{k+1}=\Proj_{X} \left(\vx^k - \vD_k^{-1}(\vg_0^k + \vh^k)\right)}; \vspace{-0.1cm}
\end{equation}
Choose $0<\rho_k\le\beta$ and update the dual variable $\vz$ by 
\begin{equation}\label{eq:update-z}
z_j^{k+1}=\left\{\begin{array}{ll}z_j^k, & \text{ if }j\neq j_k\\[0.15cm]
z_j^k + \rho_k\cdot\max\left(-\frac{z_j^k}{\beta }, f_j(\vx^{k})\right), &\text{ if }j = j_k\end{array}\right.\vspace{-0.3cm}
\end{equation}
}
\end{algorithm}

The pseudocode of the proposed method is shown in Algorithm \ref{alg:pd-sg}, which iteratively performs { (adaptive)} stochastic subgradient update to the primal variable $\vx$ and randomized coordinate update to the dual variable $\vz$. { In order to have an easy update, $\vD_k$ will be set to a diagonal matrix for each $k$. We will consider two different settings of $\vD_k$ in our analysis. \begin{setting}\label{set:nonadp}
 $\vD_k=\frac{\vI}{\alpha_k}$, where $\alpha_k>0$ for all $k$, and $\vI$ is the identity matrix. 
\end{setting}

\begin{setting}\label{set:adp}
 $\vD_k=\diag(\vs^k)+\frac{\vI}{\alpha_k}$, where $\alpha_k>0$ and $\vs^k=\eta\sqrt{\sum_{t=1}^k\frac{(\vg_0^t+\vh^t)^2}{\gamma_t^2}}$ with $\eta>0$ and $\gamma_t=\max\big(1,\|\vg_0^t+\vh^t\|\big)$ for all $t$. Here, $\va^2$ and $\sqrt{\va}$ denote the componentwise square and square-root for a vector $\va$.
 \end{setting}

Note that in Setting \ref{set:adp}, $\vD_k$ is adaptive to the primal stochastic subgradient. We scale the subgradient  for technical reasons, and it is inspired by \cite{yu2017block-normalized}. With such a setting, Algorithm \ref{alg:pd-sg} is an adaptive primal-dual stochastic gradient method, and it appears to be the first one under the primal-dual setting. Although the same order of convergence rate will be shown for both settings, the adaptive one can numerically perform significantly better.

}

We remark that if the potential application has any affine equality constraint $\va^\top \vx = b$, we can always write it into two affine inequality constraints $\va^\top\vx \le b$ and $-\va^\top\vx \le -b$ and thus formulate the problem in the form of \eqref{eq:ccp}, or we can use a technique similar to that in \cite{xu2017-1st-method-alm} to handle the equality and inequality constraints simultaneously. Furthermore, instead of sampling one constraint function every time, we can sample a small set $J_k$ of constraint functions, and let
$$\vh^k=\frac{1}{|J_k|}\sum_{j\in J_k}[\beta  f_{j}(\vx^k)+z_{j}^k]_+\tilde\nabla f_{j}(\vx^k)$$
in the update \eqref{eq:update-x} and also update $z_j$ for all $j\in J_k$. All our convergence results can still be obtained.

\subsection{Motivating examples}
We give a few examples that can be written in the form of \eqref{eq:ccp} with a very big $M$, and our proposed algorithm can be applied.


\textbf{Stochastic linear programming.}
A two-stage stochastic linear program (c.f. \cite[Sec. 2.1]{shapiro2009lectures}) can be formulated as
\begin{equation}\label{eq:2sg-slp}
\min_\vx \vc^\top \vx + \EE\big[f_\xi(\vx)\big], \st \vA\vx\le\vb,
\end{equation}
where $\xi=(\vB,\vC,\vd, \vq)$ and $f_\xi(\vx)$ are respectively the data and the optimal value of the second stage linear program
\begin{equation*}
\min_\vy \vq^\top \vy, \st \vB\vx + \vC\vy \le \vd.
\end{equation*}
As there are $M$ scenarios in the second stage, i.e., $\xi\in \{\xi_1,\ldots,\xi_M\}$ with $\Prob(\xi=\xi_i)=p_i>0$ and $\sum_{i=1}^M p_i=1$, then $$\EE\big[f_\xi(\vx)\big]=\sum_{i=1}^M p_i f_{\xi_i}(\vx)=\sum_{i=1}^M p_i \min \big\{\vq_i^\top \vy: \vB_i\vx + \vC_i\vy \le \vd_i\big\}.$$ 
Hence, \eqref{eq:2sg-slp} can be written as a single large-scale linear program:
\begin{equation}\label{eq:1single-lp}
\begin{aligned}
\min_{\vx,\vy_1,\cdots,\vy_M} & \vc^\top\vx +\sum_{i=1}^M p_i \vq_i^\top \vy_i,\
\st  \vA\vx \le \vb, \ \vB_i\vx + \vC_i\vy_i \le \vd_i,\, i=1,\ldots, M.
\end{aligned}
\end{equation}
Clearly, \eqref{eq:1single-lp} is in the form of \eqref{eq:ccp}, and if there are many scenarios, i.e., $M$ is big, it could be extremely expensive to access all the data at every update to the variables.

\textbf{Chance constrained problems by sampling and discarding.} A nonlinear program with chance constraint is formulated as 
\begin{equation}\label{eq:chance-prob}
\min_{\vx\in X} f_0(\vx), \st \Prob\big(g(\vx;\xi) \le 0\big) \ge 1-\tau,
\end{equation}
where $X\subseteq \RR^n$ is a convex set, $\xi$ is an uncertain parameter on a support set $\Xi$, and $\tau$ is a user-specified risk level of constraint violation. Even though $g(\cdot\,;\xi)$ is convex for any $\xi\in\Xi$, the chance constraint set in \eqref{eq:chance-prob} may not be convex. Hence, exactly solving \eqref{eq:chance-prob} is hard in general. To numerically solve \eqref{eq:chance-prob}, the work \cite{campi2011sampling} introduces a sample-based approximation method, called \emph{sampling and discarding} approach. This method makes $N$ independent samples of $\xi$, then eliminates $p$ of them, and solves a deterministic problem with the remaining $M=N-p$ constraints, i.e., 
\begin{equation}\label{eq:chance-prob-smpl}
\min_{\vx\in X} f_0(\vx), \st g(\vx;\xi_i) \le 0, \forall i=1,\ldots,M, 
\end{equation}
where $\{\xi_1,\ldots,\xi_M\}$ contains the $M$ samples after discarding. It is shown that under certain assumptions, for any $\vareps\in (0,1)$, if 
\begin{equation}\label{eq:cond-Np}\left(\begin{array}{c}p+n-1\\ p\end{array}\right)\sum_{i=0}^{p+n-1}\left(\begin{array}{c}N\\ i\end{array}\right)\tau^i(1-\tau)^{N-i}\le \vareps,
\end{equation}
the solution of \eqref{eq:chance-prob-smpl} is feasible for \eqref{eq:chance-prob} with probability at least $1-\vareps$.

Note that if no \emph{discarding} is performed, the above method is similar to the scenario approximation approaches in \cite{nemirovski2006scenario, luedtke2008sample}. For high-dimensional problems, i.e., $n$ is big, it is required to set a significantly bigger $N$ and also $N-p$ to have \eqref{eq:cond-Np}. Therefore, the sample-based approximation problem \eqref{eq:chance-prob-smpl} will have many functional constraints and be in the form of \eqref{eq:ccp}.

\textbf{Robust optimization by sampling.} Different from the chance constrained problem \eqref{eq:chance-prob}, robust optimization requires the constraint $g(\vx;\xi)\le 0$ to be satisfied for any $\xi\in \Xi$, i.e., 
\begin{equation}\label{eq:robust-prob}
\min_{\vx\in X} f_0(\vx), \st g(\vx;\xi) \le 0,\, \forall \xi\in \Xi.
\end{equation}
Similar to the scenario approximation method for chance constrained problems, the sampling approach (e.g., \cite{calafiore2005uncertain}) has also been proposed to numerically solve \eqref{eq:robust-prob}. Let $\{\xi_1,\ldots,\xi_M\}$ be $M$ independently extracted samples. It is shown in \cite{calafiore2005uncertain} that for any $\tau\in (0,1)$ and any $\vareps\in (0,1)$, if the number of samples satisfies 
$M\ge \frac{n}{\tau \vareps} -1,$
then the solution to \eqref{eq:chance-prob-smpl} will be a $\tau$-level robustly feasible solution with probability at least $1-\vareps$. If $n$ is big, and high feasibility level and high probability are required, then $M$ would be a very big number, and thus \eqref{eq:chance-prob-smpl} has an extremely big number of functional constraints. 


\subsection{Existing methods}
In this subsection, we review a few existing methods that could potentially be applied to solve \eqref{eq:ccp} and show how our method relates to them. Some of these methods are primal-dual type as our method, and others are purely primal methods.

\textbf{Stochastic mirror-prox method.} The proposed method is closely related to the stochastic mirror-prox method \cite{juditsky2011solving, baes2013randomized} for saddle-point problems or more generally for variational inequality (VI) problems. By the augmented Lagrangian function, one can equivalently formulate \eqref{eq:ccp} into the following saddle-point problem (c.f., \cite{rockafellar1973dual}):
\begin{equation}\label{eq:sdl-prob}
\min_{\vx\in X}\max_{\vz} \cL_\beta(\vx,\vz).
\end{equation}
Assuming $\nabla \cL_\beta$ to be Lipschitz continuous and $\vz$ in a compact set $Z$, then we can apply the method in \cite{baes2013randomized} to the above saddle-point problem and have the update:\footnote{Here, we use the Euclidean norm square as the proximal term, while \cite{baes2013randomized} actually uses a more general Bregman distance function.}
\begin{subequations}\label{eq:sd-update}
\begin{align}
%
(\hat\vx^k,\hat{\vz}^k)=&~\Proj_{X\times Z}\Big(\big(\vx^k-\alpha_k\vg^k_x, \vz^k-\alpha_k\vg^k_z\big)\Big),\label{eq:sd-update-hat}\\
(\vx^{k+1},\vz^{k+1})=&~\Proj_{X\times Z}\Big(\big(\vx^k -\alpha_k\hat{\vg}^k_x,\vz^k-\alpha_k\hat{\vg}^k_z\big)\Big),\label{eq:sd-update-2}
\end{align}
\end{subequations}
where $(\vg^k_x,\vg^k_z)$ and $(\hat\vg^k_x,\hat\vg^k_z)$ are stochastic approximation of $\nabla\cL_\beta$ at $(\vx^k,\vz^k)$ and $(\hat\vx^k,\hat\vz^k)$. The above update performs two stochastic gradient (SG) projections. To have convergence, it seems to be required for VI problems. However, for saddle-point problems, \cite{nemirovski2009robust} shows that one SG projection is sufficient for convergence guarantee, namely, simply set $(\hat\vx^k,\hat\vz^k)=(\vx^k,\vz^k)$ and then obtain $(\vx^{k+1},\vz^{k+1})$ by \eqref{eq:sd-update-2}.

The methods in \cite{baes2013randomized} and \cite{nemirovski2009robust} both require the dual variable to be in a compact set for convergence guarantee. Generally, it is difficult to estimate a valid bound on the dual variable, especially for a stochastic program. 
In addition, at each iteration, they use the same step size for both primal and dual variable update, which seems to be required in their analysis. On the contrary, 
we will not assume boundedness of $\vz$ but instead we can prove the boundedness of the sequence $\{\vz^k\}$ in expectation. Furthermore, we allow to use different step sizes, and { this is crucial for the convergence analysis of our adaptive method.} 

The SGM for saddle-point problems is also studied in \cite{palaniappan2016stochastic}. However, it requires strong convexity for both primal and dual variables. { For bilinear convex-concave saddle-point problems, the authors of \cite{chen2014optimal} give an optimal primal-dual SGM. Without assuming boundedness on either primal or dual variables, they show an $O(1/\sqrt {k})$ convergence rate in terms of a perturbed primal-dual gap, c.f. \cite[Corollary 3.4]{chen2014optimal}. Applying their method, i.e., \cite[Algorithm 3]{chen2014optimal}, to an affinely constrained convex problem, one can show that if the primal variable and the output dual iterate are bounded, then the convergence rate is $O(1/\sqrt{k})$ in terms of both primal-dual objective gap and feasibility violation.}

\textbf{Cooperative stochastic approximation.} The problem \eqref{eq:ccp} can also be equivalently formulated as a stochastic program with a single finite-sum constraint:
\begin{equation}\label{eq:ccp-csa}
\min_{\vx\in X} f_0(\vx), \st \frac{1}{M}\sum_{j=1}^M [f_j(\vx)]_+ \le 0,
\end{equation} 
and we can apply the cooperative stochastic approximation (CSA) method in \cite{lan2016algorithms-exp-cont} to find an approximate solution. At each iteration $k$, CSA first samples one constraint function $f_{j_k}$ and check its value at the iterate $\vx^k$. If $f_{j_k}(\vx^k)\ge \eta_k$, set $\vg^k= \tilde{\nabla} f_{j_k}(\vx^k)$, and otherwise, set $\vg^k$ to an unbiased estimate of $\tilde{\nabla}f_0(\vx^k)$, where $\eta_k>0$ is a parameter to control constraint violation. Then it updates the iterate by
\begin{equation}\label{eq:csa-update}
\vx^{k+1}=\Proj_X(\vx^k-\alpha_k \vg^k),
\end{equation}
where $\alpha_k$ is a step size.

For convex problems, CSA is shown to enjoy $O(1/\sqrt{k})$ convergence rate in terms of both objective and feasibility. The order can be improved to $O(1/k)$ if both the objective and constraint functions in \eqref{eq:ccp-csa} are strongly convex. We will show that the proposed algorithm can enjoy the same order of convergence rate for convex problems. To have an improved rate of $O\big((\log k)/ k\big)$, we need strong convexity of the objective function but only convexity on the constraint functions. However, we need an additional assumption on the existence of a primal-dual solution. { Hence, our method has better convergence rate for the problem with a strongly convex objective but only convex constraint functions, such as finding the projection onto the intersection of many polyhedral sets \cite{pang2015set-intersection, stovsic2016projection}.}

\textbf{Stochastic subgradient with random constraint projection.}
Let $X_0=X$ and 
\begin{equation}\label{eq:def-Xj}
X_j=\{\vx\in\RR^n: f_j(\vx)\le 0\}, \,  j=1,\ldots,M.
\end{equation} Then \eqref{eq:ccp} can be written to
\begin{equation}\label{eq:ccp-proj}
\min_\vx f_0(\vx), \st \vx \in \cX=\cap_{j=0}^M X_j.
\end{equation} 
On solving the above problem, we can apply the method in \cite{wang2015random-multi-proj, wang2016stochastic-proj} and iteratively perform the update:
\begin{equation}\label{eq:alg-proj}
\vx^{k+1}=\Proj_{X_{j_k}}\left(\vx^k-\alpha_k\vg_0^k\right),
\end{equation} 
where $j_k$ is randomly chosen from $\{0,1,\ldots,M\}$, $\Proj_{X_j}$ denotes the projection onto $X_j$, and $\vg_0^k$ is a stochastic approximation of a subgradient of $f_0$ at $\vx^k$. Various sampling schemes on $j_k$ are studied in \cite{wang2016stochastic-proj}. Under the linear regularity assumption on the set collection $\{X_j\}_{j=0}^M$, a sublinear convergence result is established. If $f_0$ is convex, the rate is $O(1/\sqrt{k})$ in terms of objective error { $|f_0(\vx^k)-f_0^*|$ and $O\big((\log k)/k\big)$ in terms of constraint violation $[\dist(\vx^k, \cX)]^2$. In \cite{wang2015random-multi-proj}, the rate of constraint violation is improved to $O(1/k)$. Furthermore, if $f_0$ is strongly convex, \cite{wang2015random-multi-proj} shows the convergence rate $O((\log k)/k)$ in terms of objective error and $O(1/k^2)$ of constraint violation.} To have efficient computation in the update \eqref{eq:alg-proj}, $X_j$ is  required to be a simple set for each $j=0,1,\ldots,M$. Hence, if $\Proj_{X_j}$'s are difficult to evaluate, such as the logistic loss function induced constraint set in the Neyman-pearson classification problem \cite{rigollet2011neyman}, the method in \cite{wang2015random-multi-proj, wang2016stochastic-proj} will be inefficient. By contrast, our update in \eqref{eq:update-x} can be computed efficiently as long as $X$ is simple. 

\textbf{Stochastic proximal-proximal gradient method.} Let $r(\vx)=\iota_X(\vx)$ and $g_j(\vx) = \iota_{X_j}(\vx)$, where $\iota_X$ denotes the indicator function on $X$, and $X_j$'s are defined in \eqref{eq:def-Xj}. Then \eqref{eq:ccp} is equivalent to
\begin{equation}\label{eq:sppg-form}
\min_\vx r(\vx) + \frac{1}{M}\sum_{j=1}^M \left(f_0(\vx)+g_j(\vx)\right).
\end{equation}
When $f_0$ is differentiable, the stochastic proximal-proximal gradient (S-PPG) method \cite{ryu2017ppg} can be applied to find a solution of \eqref{eq:sppg-form}. It starts from $(\vx^0,\vz_1^0,\ldots,\vz_M^0)$ and iteratively performs the update:
\begin{equation}\label{eq:sppg-update}
\begin{aligned}
&\vx^{k+\frac{1}{2}}=\Proj_X\left(\frac{1}{M}\sum_{j=1}^M \vz_j^k\right),\\
&\vx^{k+1}=\Proj_{X_{j_k}}\left(2\vx^{k+\frac{1}{2}}-\vz_{j_k}^k-\alpha \nabla f_0(\vx^{k+\frac{1}{2}}\right),\\
& \vz_j^{k+1}=\left\{\begin{array}{ll}\vz_j^k+\vx^{k+1}-\vx^{k+\frac{1}{2}}, &\text{ if }j=j_k\\[0.1cm]
\vz_j^k, &\text{ if }j\neq j_k\end{array}\right.
\end{aligned}
\end{equation}
where $j_k$ is chosen from $\{1,\ldots,M\}$ uniformly at random.
Since $\Proj_{X_{j_k}}$ needs be evaluated, S-PPG has the same issue as the update in \eqref{eq:alg-proj}. However, it could be more suitable in a distributed system, for which communication cost is a main concern.

\textbf{Stochastic subgradient with single projection.} Let $h(\vx)=\max_{1\le j\le M} f_j(\vx)$. Then \eqref{eq:ccp} is equivalent to 
\begin{equation}\label{eq:signle-proj}
\min_{\vx\in X} f_0(\vx), \st h(\vx)\le 0.
\end{equation}
For solving the above problem, we can apply the method in \cite{mahdavi2012stochastic}, which, at every iteration, inquires a stochastic subgradient of $f_0$ and also a subgradient of $h$. Although the method in \cite{mahdavi2012stochastic} only needs to perform a single projection to the feasible set at the last step, computing the subgradient of $h$ would generally require evaluating the function value of all $f_j$'s, and thus it is inefficient for the big-$M$ case. This issue is partly addressed in \cite{cotter2016light-touch}, which only checks a batch of randomly sampled constraint functions at every iteration. However, depending on the underlying problem and required accuracy, the batch size could be as large as $M$.

\textbf{Deterministic primal-dual first-order method.}
Other related methods are the deterministic primal-dual first-order algorithms in the author's previous works \cite{xu2017-1st-method-alm, xu2017ialm}. {  Although \cite{xu2017-1st-method-alm, xu2017ialm} also use the classic augmented Lagrangian function, their algorithm design and targeted applications are fundamentally different from those in this paper.} The methods in \cite{xu2017-1st-method-alm, xu2017ialm} assumes differentiability of $f_j$'s, and it requires exact gradient of $f_0$ and uses all $f_j,\, j=1,\ldots, M$ to update $\vx$ and $\vz$. Hence, if exact gradient of $f_0$ is not available or very expensive to compute, or if $M$ is extremely big, the deterministic methods are either inapplicable or inefficient. { In addition, the update to $\vx$ and $\vz$ in Algorithm \ref{alg:pd-sg} is Jacobi-type while \cite{xu2017-1st-method-alm, xu2017ialm} and all existing works about deterministic augmented Lagrangian method update the primal and dual variables in a Gauss-Seidel manner. Furthermore, due to the stochasticity, the analysis of this paper is fundamentally different and more complicated than that in \cite{xu2017-1st-method-alm, xu2017ialm}.} Similarly, the deterministic first-order methods in \cite{yu2016primal, yu2017primal-composite, lin2018level-set} are also very expensive or do not apply for the stochastic program with many constraints. 

Besides the above reviewed methods, in the literature there are also other methods that can be applied to \eqref{eq:ccp} such as the penalty method with stochastic approximation \cite{lan2016algorithms-exp-cont}. Exhausting all the existing methods is impossible. We refer the interested readers to the papers above and the references therein. 

\subsection{Contributions}
The main contributions are listed below.
\begin{itemize}[leftmargin=*]
\item We propose a novel { (adaptive)} primal-dual SGM for solving stochastic programs with many functional constraints. The method is derived based on the classical augmented Lagrangian function. Through a stochastic oracle, it alternatingly performs stochastic subgradient update to the primal variable and randomized coordinate update to the dual variable. At each iteration, it only needs to sample one out of many constraint functions and thus has low per-iteration complexity. 

\item We establish convergence rate results of the proposed method for convex problems and also problems with strongly convex objective. Different from existing analysis of primal-dual SGM for saddle-point problems, we do not assume the boundedness of the dual variable $\vz$, but instead we prove the boundedness of the dual iterate in expectation. For convex problems, we show that the algorithm can achieve the optimal $O(1/\sqrt{k})$ convergence rate, and for problems with strongly convex objective, we show that it can achieve $O\big((\log k)/k\big)$ convergence rate, where $k$ is the number of subgradient inquiries. All convergence rate results are in terms of primal and/or dual objective value and also primal constraint violation. For the strongly convex case, the $\log k$ factor can be removed if the dual iterate sequence is assumed to be bounded; see Remark \ref{rm:scvx-rate}. To the best of our knowledge, no existing work has established $O(1/k)$ convergence rate result for a primal-dual SGM by assuming strong convexity only on the primal objective function, even if the dual variable is restricted in a bounded set. The CSA method in \cite{lan2016algorithms-exp-cont} is a primal SGM, and it has $O(1/k)$ convergence rate if both the objective and constraint functions are strongly convex. 

\item We show the practical performance of the proposed algorithm by testing it on solving a sample approximation problem of the robust portfolio selection and convex quadratically constrained quadratic programs. The numerical results demonstrate that the proposed primal-dual SGM can be significantly better than the stochastic mirror-prox algorithm in \cite{baes2013randomized} and the CSA method in \cite{lan2016algorithms-exp-cont}.
\end{itemize}

\subsection{Notation and outline}
We use bold lower-case letters $\vx,\vz,\ldots$ for vectors and $x_i, z_i,\ldots$ for their $i$-th components. The bold number $\vzero$ and $\vone$ denote the all-zero and all-one vectors, respectively. $[M]$ is short for the set $\{1,2,\ldots,M\}$, $[a]_+=\max(0,a)$ and $[a]_-=\max(0,-a)$ respectively denote the positive and negative parts of a real number $a$. Given a symmetric positive semidefinite matrix $\vD$, $\|\vx\|_\vD$ is defined as $\sqrt{\vx^\top\vD\vx}$. We use $\|\vx\|$ to denote the Euclidean norm of a vector $\vx$. For two vectors $\vx$ and $\vy$ of the same size, $\vx\odot\vy$ denotes their componentwise product. 
For a convex function $f$, we denote by $\tilde{\nabla}f(\vx)$ a subgradient of $f$ at $\vx$, and the set of all subgradients of $f$ at $\vx$ is called the subdifferential of $f$, denoted by $\partial f(\vx)$. For a closed convex set $X$, $\Proj_X$ denotes the projection operator onto $X$. We let $\cH^k$ contain the history of Algorithm \ref{alg:pd-sg} until $(\vx^k,\vz^k)$, i.e.,
 $\cH^k=\big\{\vx^1, \vz^1, \vx^2,\vz^2 \ldots, \vx^{k}, \vz^{k}\big\}
 .$
$\EE[\zeta]$ denotes the expectation of a random variable $\zeta$, and $\EE[\zeta\,|\,\xi]$ is for the expectation of $\zeta$ conditional on $\xi$. 
In addition, we denote 
 \begin{equation}\label{eq:def-Phi}
 \textstyle\Phi(\bar{\vx}; \vx,\vz) = f_0(\bar{\vx})-f_0(\vx) + \frac{1}{M}\sum_{j=1}^M z_j f_j(\bar{\vx}).
 \end{equation}
 
The rest of the paper is outlined as follows. In section \ref{sec:assump}, we give the technical assumptions required in our analysis, and in section \ref{sec:analysis}, we analyze the algorithm with nonadaptive setting and show its convergence rate results. The convergence rate result of the algorithm with adaptive setting is given in section \ref{sec:analysis-adp}. Numerical results are provided in section \ref{sec:numerical}, and finally section \ref{sec:conclusion} concludes the paper.

\section{Technical assumptions}\label{sec:assump}

Throughout our analysis, we make the following assumptions. 

\begin{assumption}\label{assump:kkt}
There exists a primal-dual solution $(\vx^*, \vz^*)$ satisfying the Karush-Kuhn-Tucker (KKT) conditions:
\begin{subequations}\label{eq:kkt-conds}
\begin{align}
&\vzero\in \partial f_0(\vx^*)+\cN_X (\vx^*)+\frac{1}{M}\sum_{j=1}^M z_j^* \partial f_j(\vx^*),\label{eq:df}\\
&\vx^*\in X,\ f_j(\vx^*)\le 0, \forall j\in [M],\label{eq:pf}\\
& z_j^*\ge 0, \, z_j^* f_j(\vx^*) = 0, \forall j\in [M],\label{eq:cs}
\end{align}
\end{subequations}
where $\cN_X(\vx)$ denotes the normal cone of $X$ at $\vx$.
\end{assumption}

\begin{assumption}\label{assump:bd}
The SG approximation $\vg_0^k$ is unbiased and bounded, i.e., there is a constant $\sigma > 0$ such that
$$\EE\big[\vg_0^k\,\big |\, \cH^k\big] \in \partial f_0(\vx^k),\quad \EE\big[\|\vg_0^k\|^2\,\big |\, \cH^k\big] \le \sigma^2, \,\forall k.$$ 
In addition, there exist constants $F$ and $G$ such that
$$|f_j(\vx)|\le F, \, \|\tilde{\nabla} f_j(\vx)\| \le G, \,\forall\, \tilde{\nabla} f_j(\vx) \in \partial f_j(\vx),\,\forall j\in [M],\, \forall \vx\in X.$$
\end{assumption}

\begin{assumption}\label{assump:cvx}
For each $j=0,1,\ldots,M$, $f_j$ is a closed convex function on $X$. In addition, $f_0$ is $\mu$-strongly convex, i.e., 
\begin{equation}\label{eq:str-cvx-ineq}
f_0(\vy)\ge f_0(\vx) + \langle \tilde{\nabla}f_0(\vx), \vy-\vx\rangle +\frac{\mu}{2}\|\vy-\vx\|^2, \forall \vx, \vy\in X.
\end{equation}
\end{assumption}

Assumption \ref{assump:kkt} is satisfied if a certain constraint qualification holds such as the Slater's condition \cite{bazaraa2013nonlinear}. In Assumption \ref{assump:bd}, the unbiasedness and boundedness assumption on $\vg_0^k$ is standard in the literature of SGM, and the boundedness of each $f_j$ and $\tilde{\nabla} f_j$ is satisfied if $X$ is bounded. In Assumption \ref{assump:cvx}, if $\mu=0$, then $f_0$ is simply a convex function.

As the KKT conditions in \eqref{eq:kkt-conds} hold, there are $\tilde{\nabla}f_j(\vx^*),\,\forall j\in [M]$ such that
$$\textstyle-\frac{1}{M}\sum_{j=1}^M z_j^* \tilde{\nabla}f_j(\vx^*)\in\partial f_0(\vx^*)+\cN_X(\vx^*).$$
Hence, from the convexity of $f_0$ and $X$, it follows that
\begin{equation}\label{eq:kkt-ineq-pf}
\textstyle f_0(\vx) \ge f_0(\vx^*) - \left\langle\frac{1}{M}\sum_{j=1}^M z_j^* \tilde{\nabla}f_j(\vx^*), \vx-\vx^*\right\rangle,\,\forall \vx\in X.
\end{equation}
Since $z_j^*\ge 0$ and $f_j$ is convex for each $j\in [M]$, we have
$$z_j^*\big(f_j(\vx)-f_j(\vx^*)\big)\ge \langle z_j^* \tilde{\nabla}f_j(\vx^*), \vx-\vx^* \rangle.$$
The above inequality together with \eqref{eq:kkt-ineq-pf} and the fact $z_j^*f_j(\vx^*) = 0,\,\forall j\in [M]$ implies
\begin{equation}\label{eq:opt-cond}
\Phi(\vx;\vx^*,\vz^*)=f_0(\vx)-f_0(\vx^*)+\frac{1}{M}\sum_{j=1}^M z_j^* f_j(\vx) \ge 0, \forall \vx\in X.
\end{equation}

Furthermore, note that for any $\beta>0$, it holds $[\beta f_j(\vx^*)+z_j^*]_+ = z_j^*,\,\forall j\in [M]$, and thus \eqref{eq:df} exactly means $\vzero\in\partial_\vx \cL_\beta(\vx^*,\vz^*)+\cN_X(\vx^*)$. Hence, $\vx^*$ is a solution of $\min_{\vx\in X}\cL_\beta(\vx,\vz^*)$, which indicates 
$d_\beta(\vz^*)=\cL_\beta(\vx^*,\vz^*).$ From the definitions of $\Psi_\beta$ and $\psi_\beta$ in \eqref{eq:def-Psi} and \eqref{eq:psi}, and also \eqref{eq:pf} and \eqref{eq:cs}, it is straightforward to have $\Psi_\beta(\vx^*,\vz^*)=0$. Therefore,
\begin{equation}\label{eq:str-dual}
d_\beta(\vz^*)=f_0(\vx^*),
\end{equation}
i.e., the strong duality holds, and $\vx^*$ and $\vz^*$ are primal and dual optimal solutions.

\section{Convergence analysis of the nonadaptive method}\label{sec:analysis}
For ease of understanding, we first analyze the convergence of Algorithm \ref{alg:pd-sg} with the nonadaptive Setting \ref{set:nonadp}. Under Assumptions \ref{assump:kkt} through \ref{assump:cvx}, we show that for convex problems, our method can achieve the optimal convergence rate $O(1/\sqrt{k})$, and for problems with strongly convex objective, it can achieve a near-optimal rate $O((\log k)/k)$, where $k$ is the number of iterations. While existing analysis \cite{nedic2009subgradient, baes2013randomized} for saddle-point problems assumes the boundedness of the dual variable, we do not require such an assumption. Instead we can bound all $\vz^k$ in expectation by choosing appropriate parameters.  In addition, we do not find any existing work that has shown $O((\log k)/k)$ rate for a primal-dual SGM by assuming strong convexity on the primal objective.  
 
 \subsection{Preliminary results}
 We first establish a few preliminary results. 
The lemma below can be directly verified from the definition of $\Psi_\beta$.
\begin{lemma}\label{lem:psi-neg}
Let $\beta>0$. Then for any $\vx\in X$ such that $f_j(\vx)\le 0, \,\forall j\in [M]$ and any $\vz\ge \vzero$, it holds $\Psi_\beta(\vx, \vz) \le 0$.
\end{lemma}

The next lemma is important to establish the convergence rate of our algorithm. Similar ones { in a deterministic form} have appeared in \cite{xu2017ialm, xu2017-1st-method-alm}. 

\begin{lemma}\label{lem:pre-rate}
Let $\bar{\vx}\in X$ and $\bar \vz$ be random vectors, and let $\vareps_1\ge0$ and $\vareps_2\ge0$ be scalars. If for any $\vx\in X$ and $\vz\ge \vzero$ that may depend on $(\bar{\vx},\bar \vz)$, it holds
\begin{equation}\label{eq:ergo-0}
\EE\left[f_0(\bar\vx)+\frac{1}{M}\sum_{j=1}^M z_j f_j(\bar\vx)\right]\le \EE\big[f_0(\vx)+\Psi_\beta(\vx,\bar\vz)\big] + \vareps_1 + \vareps_2 \EE\|\vz\|^2,
\end{equation}
then for any $(\vx^*,\vz^*)$ satisfying \eqref{eq:kkt-conds}, 
\begin{align}
&\EE\big|f_0(\bar{\vx})-f_0(\vx^*)\big|\le 2\vareps_1+9\vareps_2 \|\vz^*\|^2,\label{eq:bd-obj}\\
&\EE\left[\frac{1}{M}\sum_{j=1}^M [f_j(\bar{\vx})]_+\right] \le \vareps_1 + \vareps_2 \|\vone+\vz^*\|^2,\label{eq:bd-res}\\
&\EE\big[d_\beta(\vz^*)-d_\beta(\bar\vz)\big] \le \frac{3}{2}\big(\vareps_1+3\vareps_2\|\vz^*\|^2\big).\label{eq:bd-dual}
\end{align}
\end{lemma}
\begin{proof}
Let $\vx=\vx^*$ in \eqref{eq:ergo-0} and recall the definition of $\Phi$ in \eqref{eq:def-Phi}. Then by Lemma \ref{lem:psi-neg}, we have 
\begin{equation}\label{eq:ergo}
\EE\big[\Phi(\bar{\vx};\vx^*,\vz)\big] \le \vareps_1 + \vareps_2 \EE\|\vz\|^2.
\end{equation}

Since $-z_j^* f_j(\bar{\vx}) \ge -z_j^* [f_j(\bar{\vx})]_+$, we have from \eqref{eq:opt-cond} that
\begin{equation}\label{eq:bd-diff-f0}
f_0(\bar{\vx})-f_0(\vx^*) \ge -\frac{1}{M}\sum_{j=1}^M z_j^* [f_j(\bar{\vx})]_+.
\end{equation}
We obtain the inequality in \eqref{eq:bd-res}, by substituting the above inequality into \eqref{eq:ergo} with $\vz$ given by $z_j=1+z_j^*$ if $f_j(\bar{\vx}) >0$ and $z_j=0$ otherwise for any $j\in [M]$.

Letting $z_j = 3z^*_j$ if $f_j(\bar{\vx})>0$ and $z_j=0$ otherwise for each $j\in [M]$ in \eqref{eq:ergo} and adding \eqref{eq:bd-diff-f0} together gives
\begin{equation}\label{eq:sum-z-fi}
\EE\left[\frac{1}{M}\sum_{j=1}^M z_j^* [f_j(\bar{\vx})]_+\right]\le \frac{\vareps_1}{2}+\frac{9\vareps_2}{2}\|\vz^*\|^2.
\end{equation}
Hence, by the above inequality and \eqref{eq:bd-diff-f0}, we obtain 
$
\EE\big[f_0(\bar{\vx})-f_0(\vx^*)\big]_-\le \frac{\vareps_1}{2}+\frac{9\vareps_2}{2}\|\vz^*\|^2.
$
In addition, from \eqref{eq:ergo} with $\vz=\vzero$, it follows $\EE[f_0(\bar{\vx})-f_0(\vx^*)]\le \vareps_1$. Since $|a|= a+2[a]_-$ for any real number $a$, we have 
$$\EE\big|f_0(\bar{\vx})-f_0(\vx^*)\big| = \EE[f_0(\bar{\vx})-f_0(\vx^*)]+2\EE\big[f_0(\bar{\vx})-f_0(\vx^*)\big]_-\le 2\vareps_1+9\vareps_2\|\vz^*\|^2,$$
which gives \eqref{eq:bd-obj}.

Furthermore, in \eqref{eq:ergo-0}, let $\vz=\vzero$ and take 
$\vx\in\argmin_{\vx\in X} f_0(\vx)+\Psi_\beta(\vx,\bar\vz).$ We have
$\EE f_0(\bar\vx)\le \EE d_\beta(\bar\vz)+\vareps_1,$ which together with \eqref{eq:bd-diff-f0}, \eqref{eq:sum-z-fi}, and \eqref{eq:str-dual} gives the inequality in \eqref{eq:bd-dual}.
\end{proof}

\begin{remark}\label{rm:rate-obj-res}{lem:pre-rate}
From the proof of Lemma \ref{lem:pre-rate}, we see that if \eqref{eq:ergo} holds for any $\vz\ge\vzero$, then the inequalities in \eqref{eq:bd-obj} and \eqref{eq:bd-res} hold.
\end{remark}

The following two lemmas will be used to establish an important inequality for running one iteration of Algorithm \ref{alg:pd-sg}. Their proofs are given in the appendix.
\begin{lemma}\label{lem:z-term}
For any { deterministic or stochastic} $\vz\ge\vzero$, it holds
\begin{align}\label{eq:Psi-term}
&~-\Psi_{\beta }(\vx^{k},\vz^k)+\frac{1}{M}\sum_{j=1}^M z_j f_j(\vx^{k}) + \frac{1}{2\rho_k}\EE\left[ \|\vz^{k+1}-\vz\|^2\,\big|\, \cH^{k}\right]\cr
\le&~ \frac{1}{2\rho_k} \|\vz^{k}-\vz\|^2 { -\frac{1}{2\rho_k}\left(\frac{\beta }{\rho_k}-1\right)\EE\left[\|\vz^{k+1}-\vz^k\|^2\,\big|\, \cH^{k}\right]} \\
&~{ + \EE\left[\big\langle\vz^k-\vz, M\ve_{j_k}\odot \nabla_\vz \Psi(\vx^k,\vz^k) - \nabla_\vz \Psi(\vx^k,\vz^k) \big\rangle\,\big|\, \cH^{k}\right]}.\nonumber 
\end{align}
\end{lemma}

\begin{lemma}\label{lem:bd-psi-term}
Under Assumption \ref{assump:bd}, for any $\vx\in X$ and any $\vz$, it holds
\begin{equation}\label{eq:bd-psi-term}
\frac{1}{M}\sum_{j=1}^M \|\tilde{\nabla}_\vx \psi_{\beta}(f_j(\vx), z_j)\|^2 \le 2\beta^2 F^2 G^2 + \frac{2G^2}{M}\|\vz\|^2.
\end{equation}
\end{lemma}

By the previous three lemmas, we establish an important result for running one iteration of Algorithm \ref{alg:pd-sg} and then use it to show the convergence rate results. 
\begin{theorem}[fundamental result]\label{lem:1iter-bd}
Under Assumptions \ref{assump:bd} and \ref{assump:cvx}, { and assuming $\vD_k\succeq \frac{\vI}{\alpha_k},\, \forall \, k$ for a positive number sequence $\{\alpha_k\}_{k\ge1}$,} let $(\vx, \vz)$ be any { deterministic or stochastic} vector such that $\vx\in X$ and $\vz\ge\vzero$. Then
\begin{align}\label{eq:1iter-bd}
&~\EE\left[f_0(\vx^k)+\frac{1}{M}\sum_{j=1}^M z_j f_j(\vx^k)\right]+\frac{1}{2}\EE\|\vx^{k+1}-\vx\|_{\vD_k}^2+\frac{1}{2\rho_k}\EE\|\vz^{k+1}-\vz\|^2\nonumber\\
\le &~\EE\left[f_0(\vx)+\Psi_{\beta }(\vx,\vz^k)\right]+\frac{1}{2}\EE \|\vx^{k}-\vx\|_{\vD_k-\mu\vI}^2+\frac{1}{2\rho_k}\EE\|\vz^k-\vz\|^2\\
&~+\alpha_k\left(\sigma^2 + 2\beta ^2 F^2 G^2 + \frac{2G^2}{M}\EE\|\vz^k\|^2\right) { -\frac{1}{2\rho_k}\left(\frac{\beta }{\rho_k}-1\right)\EE\|\vz^{k+1}-\vz^k\|^2}\cr
&~{ -\EE\left[\big\langle \vx^{k}-\vx, \vg_0^k-\tilde{\nabla}f_0(\vx^k)\big\rangle\right] -\EE\left[\big\langle \vx^{k}-\vx, \vh^k-\tilde\nabla_\vx\Psi_\beta(\vx^k,\vz^k)\big\rangle\right]}\cr
&~{ +\EE\left[\big\langle\vz^k-\vz, M\ve_{j_k}\odot \nabla_\vz \Psi(\vx^k,\vz^k) - \nabla_\vz \Psi(\vx^k,\vz^k) \big\rangle\right]},\nonumber
\end{align}
{ where $\tilde{\nabla}f_0(\vx^k)=\EE[\vg_0^k\,|\,\cH^k]$ and $\tilde\nabla_\vx\Psi_\beta(\vx^k,\vz^k)=\EE[\vh^k\,|\,\cH^k]$.}
\end{theorem}

\begin{proof} 
From the update \eqref{eq:update-x}, it follows that for any $\vx\in X$,
\begin{equation}\label{eq:cross-0}
\left\langle \vx^{k+1}-\vx, \vg_0^k + \vh^k+\vD_k(\vx^{k+1}-\vx^k)\right\rangle \le 0.
\end{equation}
Next we estimate a lower bound about the left hand side of the above inequality. First, We write
$
\big\langle \vx^{k+1}-\vx, \vg_0^k\big\rangle = \big\langle \vx^{k+1}-\vx^k, \vg_0^k\big\rangle+\big\langle \vx^{k}-\vx, \vg_0^k\big\rangle.
$
By the Young's inequality, it holds 
\begin{equation}\label{eq:gcross-1}
\big\langle \vx^{k+1}-\vx^k, \vg_0^k\big\rangle\ge -\frac{1}{4\alpha_k}\|\vx^{k+1}-\vx^k\|^2-\alpha_k\|\vg_0^k\|^2.
\end{equation}
Also, we write 
${ \big\langle \vx^{k}-\vx, \vg_0^k\big\rangle=\big\langle \vx^{k}-\vx, \tilde{\nabla}f_0(\vx^k)\big\rangle + \big\langle \vx^{k}-\vx, \vg_0^k-\tilde{\nabla}f_0(\vx^k)\big\rangle,}$ where $\tilde{\nabla}f_0(\vx^k)=\EE[\vg_0^k\,|\, \cH^k]\in\partial f_0(\vx^k)$. Hence, from \eqref{eq:gcross-1} and \eqref{eq:str-cvx-ineq}, it follows that
\begin{align*}
\big\langle \vx^{k+1}-\vx, \vg_0^k\big\rangle
\ge &~ -\frac{1}{4\alpha_k}\|\vx^{k+1}-\vx^k\|^2-\alpha_k\|\vg_0^k\|^2+f_0(\vx^k)-f_0(\vx)+\frac{\mu}{2}\|\vx^k-\vx\|^2\\
&~ { + \big\langle \vx^{k}-\vx, \vg_0^k-\tilde{\nabla}f_0(\vx^k)\big\rangle.}
\end{align*}
Taking conditional expectation, we have from the above inequality and Assumption \ref{assump:bd} that
\begin{align}\label{eq:cross-1}
&~\EE\left[\big\langle \vx^{k+1}-\vx, \vg_0^k\big\rangle\,\big|\, \cH^k\right]\cr
\ge&~ - \frac{1}{4\alpha_k} \EE\left[\|\vx^{k+1}-\vx^k\|^2\,\big|\, \cH^k\right] - \alpha_k\sigma^2+\EE\left[f_0(\vx^k)-f_0(\vx)+\frac{\mu}{2}\|\vx^k-\vx\|^2\,\big|\, \cH^k\right]\cr
&~{ +\EE\left[\big\langle \vx^{k}-\vx, \vg_0^k-\tilde{\nabla}f_0(\vx^k)\big\rangle\,\big|\, \cH^k\right]}.
\end{align}
Similar to \eqref{eq:cross-1}, we have
\begin{align}\label{eq:cross-2}
&~\EE\left[\big\langle \vx^{k+1}-\vx, \vh^k\big\rangle\,\big|\, \cH^k\right]\\
\ge &~-\EE\left[\frac{1}{4\alpha_k}\|\vx^{k+1}-\vx^k\|^2+\alpha_k\|\vh^k\|^2\,\big|\, \cH^k\right]+\EE\left[\Psi_{\beta }(\vx^k,\vz^k)-\Psi_{\beta }(\vx,\vz^k)\, \big|\,\cH^k\right]\cr
&~{ +\EE\left[\big\langle \vx^{k}-\vx, \vh^k-\tilde\nabla_\vx\Psi_\beta(\vx^k,\vz^k)\big\rangle\,\big|\, \cH^k\right]},\nonumber
\end{align}
where $\tilde\nabla_\vx\Psi_\beta(\vx^k,\vz^k)=\EE[\vh^k\,|\,\cH^k]$. 
Since $j_k$ is chosen from $[M]$ uniformly at random, by \eqref{eq:subgrad-psi}, \eqref {eq:bd-psi-term} and the Young's inequality, we have
\begin{align}\label{eq:bd-hk-vec}
-\alpha_k\EE\big[\|\vh^k\|^2\,\big|\, \cH^k\big]
=& ~ - \frac{\alpha_k}{M}\sum_{j=1}^M \big\|\tilde{\nabla}_\vx \psi_{\beta } \big(f_{j}(\vx^k), z_{j}^k\big)\big\|^2\cr
\ge &~ - \alpha_k\left(2\beta ^2 F^2 G^2 + \frac{2G^2}{M}\|\vz^k\|^2\right).
\end{align}
In addition,
\begin{align}\label{eq:cross-3}
\left\langle \vx^{k+1}-\vx, \vD_k(\vx^{k+1}-\vx^k)\right\rangle=\frac{1}{2}\left[\|\vx^{k+1}-\vx\|_{\vD_k}^2-\|\vx^{k}-\vx\|_{\vD_k}^2+\|\vx^{k+1}-\vx^k\|_{\vD_k}^2\right].
\end{align}

Taking expectation on both sides of \eqref{eq:cross-1} through \eqref{eq:cross-3}, summing them up, substituting into \eqref{eq:cross-0}, and noting $\vD_k\succeq \frac{\vI} {\alpha_k}$ gives 
\begin{align}\label{eq:cross-all}
&~\EE\left[f_0(\vx^{k})-f_0(\vx)+\Psi_{\beta }(\vx^k,\vz^k)-\Psi_{\beta }(\vx,\vz^k)\right]+\frac{1}{2}\EE\|\vx^{k+1}-\vx\|_{\vD_k}^2\nonumber\\
\le &~ \frac{1}{2}\EE \|\vx^{k}-\vx\|_{\vD_k-\mu\vI}^2 + \alpha_k\left(\sigma^2 + 2\beta ^2 F^2 G^2 + \frac{2G^2}{M}\EE\|\vz^k\|^2\right)\\
&~{ -\EE\left[\big\langle \vx^{k}-\vx, \vg_0^k-\tilde{\nabla}f_0(\vx^k)\big\rangle\right] -\EE\left[\big\langle \vx^{k}-\vx, \vh^k-\tilde\nabla_\vx\Psi_\beta(\vx^k,\vz^k)\big\rangle\right]}\nonumber.
\end{align}
Taking expectation on both sides of \eqref{eq:Psi-term}, adding it 
to \eqref{eq:cross-all}, and rearranging terms
yield the desired result.
\end{proof}

By Theorem \ref{lem:1iter-bd}, we can bound the growth of $\EE\|\vz^k\|^2$ as below. Its proof is given in the appendix.
\begin{proposition}\label{prop:bd-z-prelim}
Under Assumptions \ref{assump:kkt} through \ref{assump:cvx}, { and assuming $\vD_k\succeq \frac{\vI}{\alpha_k},\, \forall \, k$ for a positive number sequence $\{\alpha_k\}_{k\ge1}$,} let $\big\{(\vx^k,\vz^k)\big\}$ be the sequence generated from Algorithm \ref{alg:pd-sg} with parameters satisfying
\begin{equation}\label{eq:parameters-cond}
\rho_k\vD_k \succeq \rho_{k+1}(\vD_{k+1}-\mu\vI),\,\forall k\ge 1,
\end{equation}
then for any $t\ge1$, it holds that
\begin{align}\label{eq:bd-z-unified}
&~\EE\|\vz^{t+1}\|^2\\
\le &~2\rho_1\|\vx^1-\vx^*\|_{\vD_1-\mu\vI}^2+4\|\vz^*\|^2+\sum_{k=1}^t 4\alpha_k\rho_k\left(\sigma^2 + 2\beta ^2 F^2 G^2 + \frac{2G^2}{M}\EE\|\vz^k\|^2\right),\nonumber
\end{align}
where $(\vx^*,\vz^*)$ is any point satisfying the KKT conditions in \eqref{eq:kkt-conds}.
\end{proposition}

\subsection{Convergence rate for convex problems}
In this subsection, we establish the convergence rate of Algorithm \ref{alg:pd-sg} for convex problems, i.e., $\mu=0$. Different from existing analysis for saddle-point problems, we do not assume the boundedness of the dual variable $\vz$ but instead we can bound $\vz^k$ in expectation.

Using Proposition \ref{prop:bd-z-prelim}, we specify the parameters and bound $\EE\|\vz^k\|^2$. The proofs of both propositions below are given in the appendix.
\begin{proposition}[pre-determined maximum number of iterations]\label{prop:fix-iter}
Under Assumptions \ref{assump:kkt} through \ref{assump:cvx}, given a positive integer $K$, set  
\begin{equation}\label{eq:para-fix-iter}
{ \vD_k = \frac{\sqrt{K}}{\alpha}\vI},\, \rho_k = \frac{\rho}{\sqrt{K}},\, \beta \ge \rho, \,\forall 1\le k\le K,
\end{equation}
where $\alpha,\rho$ and $\beta$ are positive scalars satisfying $\alpha \rho < \frac{M}{8G^2}$. Then for any $1\le k\le K+1$, it holds that
\begin{equation}\label{eq:bd-z}
\EE\|\vz^k\|^2\le \frac{C_1}{1-\frac{8\alpha\rho G^2}{M}}
\end{equation}
where 
\begin{equation}\label{eq:def-C}
C_1=\frac{2\rho}{\alpha}\|\vx^{1}-\vx^*\|^2+4\|\vz^*\|^2+4\alpha\rho\big(\sigma^2+2\beta^2 F^2G^2\big).
\end{equation}
\end{proposition}

If the maximum number of iterations is not pre-determined, we set parameters adaptive to iteration numbers and can still bound $\EE\|\vz^k\|^2$.
\begin{proposition}[varying maximum number of iterations]\label{prop:nfix-iter}
Under Assumptions \ref{assump:kkt} through \ref{assump:cvx}, let $\big\{(\vx^k,\vz^k)\big\}$ be the sequence generated from Algorithm \ref{alg:pd-sg} with parameters set to
\begin{equation}\label{eq:para-nonfix-iter}
{ \vD_k = \frac{\sqrt{k+1}\log (k+1)}{\alpha}\vI},\, \rho_k = \frac{\rho}{\sqrt{k+1}\log (k+1)},\, \beta \ge\rho, \,\forall k\ge 1,
\end{equation}
where $\alpha, \rho$ and $\beta$ are positive scalars satisfying $\alpha \rho < \frac{M}{20G^2}$. Then for any $k\ge 1$, it holds that
\begin{equation}\label{eq:bd-z-nfix}
\EE\|\vz^k\|^2\le \frac{C_2}{1-\frac{20\alpha\rho G^2}{M}},
\end{equation}
where 
\begin{equation}\label{eq:def-C2}
C_2=\frac{2\rho}{\alpha}\|\vx^1-\vx^*\|^2+4\|\vz^*\|^2+10\alpha\rho\big(\sigma^2+2 \beta^2 F^2 G^2\big).
\end{equation}
\end{proposition}

{ To show the convergence rate results, we need the following lemma to handle the last three expectation terms in \eqref{eq:1iter-bd}. Its proof is given in the appendix and follows the proof of \cite[Lemma 3.1]{nemirovski2009robust}. 
\begin{lemma}\label{lem:bd-cross}
For any deterministic or stochastic vector $(\vx,\vz)$ with $\vx\in X$ and $\vz\ge\vzero$, it holds for any positive number sequence $\{\alpha_k\}$ that
\begin{align}
&\textstyle-\sum_{k=1}^K\alpha_k \EE\big\langle \vx^k-\vx, \vg_0^k-\tilde{\nabla}f_0(\vx^k)\big\rangle \le  \frac{1}{2}\EE\|\vx^1-\vx\|^2 + \frac{\sigma^2}{2} \sum_{k=1}^K\alpha_k^2,\label{eq:bd-cross-1}\\[0.2cm]
&\begin{array}{l}
-\sum_{k=1}^K\alpha_k\EE\big\langle \vx^{k}-\vx, \vh^k-\nabla_\vx\Psi_\beta(\vx^k,\vz^k)\big\rangle \le  \frac{1}{2}\EE\|\vx^1-\vx\|^2\\[0.1cm]
\hspace{5cm}+ \sum_{k=1}^K\alpha_k^2 \left(\beta ^2 F^2 G^2 + \frac{G^2}{M}\EE\|\vz^k\|^2\right),
\end{array}\label{eq:bd-cross-2}\\[0.2cm]
&\begin{array}{l}
\sum_{k=1}^K\alpha_k \EE\big\langle\vz^k-\vz, M\ve_{j_k}\odot \nabla_\vz \Psi(\vx^k,\vz^k) - \nabla_\vz \Psi(\vx^k,\vz^k) \big\rangle \\[0.1cm]
\hspace{6cm}\le \frac{1}{2}\EE\|\vz^1-\vz\|^2 + \frac{F^2}{2} \sum_{k=1}^K\alpha_k^2. 
\end{array}\label{eq:bd-cross-3}
\end{align}
\end{lemma}

}

Using Theorem \ref{lem:1iter-bd} and also the boundedness of $\EE\|\vz^k\|^2$, we are now ready to show the convergence rate results for the case $\mu=0$. First, we establish a result with constant step sizes, and the order is $O(1/\sqrt{k})$, where $k$ is the iteration number.
\begin{theorem}[Convergence rate for convex case with constant step sizes]\label{thm:rate-cvx-fix}
Under Assumptions \ref{assump:kkt} through \ref{assump:cvx}, let $\{(\vx^k, \vz^k)\}$ be the sequence generated from Algorithm \ref{alg:pd-sg}. Given any positive integer $K$, set the parameters according to \eqref{eq:para-fix-iter}, let $\bar{\vx}^{K}=\frac{1}{K}\sum_{k=1}^K\vx^{k}$ and $\bar\vz^K=\frac{1}{K}\sum_{k=1}^K\vz^{k}$, and define
\begin{equation}\label{eq:def-phi1}
\phi_1(\vx)=\frac{3}{2\alpha}\|\vx^{1}-\vx\|^2+\alpha\left(\frac{3}{2}\sigma^2 + 3\beta^2 F^2 G^2 + \frac{3G^2}{M}\frac{C_1}{1-\frac{8\alpha\rho G^2}{M}}+\frac{F^2}{2}\right),
\end{equation}
where $C_1$ is defined in \eqref{eq:def-C}.
Then
{  
\begin{subequations}\label{eq:rate-fix-iter}
\begin{align}
&\EE\big|f_0(\bar{\vx}^{K})-f_0(\vx^*)\big| \le \frac{1}{\sqrt{K}} \left(2\phi_1(\vx^*)+ \frac{9(\alpha+\rho)}{2\alpha\rho}\|\vz^*\|^2\right),\label{eq:rate-fix-iter-obj}\\
&\EE\left[\frac{1}{M}\sum_{j=1}^M[f_j(\bar{\vx}^{K})]_+\right] \le   \frac{1}{\sqrt{K}}\left(\phi_1(\vx^*) + \frac{\alpha+\rho}{2\alpha\rho}\|\vone+\vz^*\|^2\right).\label{eq:rate-fix-iter-res}
\end{align}
In addition, if $X$ is bounded, then 
\begin{equation}\label{eq:rate-fix-iter-dual}
\EE\big[d_\beta(\vz^*)-d_\beta(\bar\vz^K)\big]\le \frac{3}{2\sqrt{K}}\left(\max_{\vx\in X}\phi_1(\vx)+\frac{3(\alpha+\rho)}{2\alpha\rho}\|\vz^*\|^2\right).
\end{equation}
\end{subequations}
}
\end{theorem}

\begin{proof}
When the parameters are set according to \eqref{eq:para-fix-iter}, we have \eqref{eq:bd-z}. 
Hence, multiplying $\alpha_k=\frac{\alpha}{\sqrt K}$ to \eqref{eq:1iter-bd}, summing it up from $k=1$ through $K$, using \eqref{eq:bd-cross-1} through \eqref{eq:bd-cross-3}, and noting $\vz^1=\vzero$ give
\begin{align*}
\frac{\alpha}{\sqrt{K}}\sum_{k=1}^K \EE\left[f_0(\vx^k)+\frac{1}{M}\sum_{j=1}^M z_j f_j(\vx^k)\right]\le &~\frac{\alpha}{\sqrt{K}}\sum_{k=1}^K\EE\left[f_0(\vx)+\Psi_{\beta}(\vx,\vz^k)\right]\nonumber\\
&\hspace{-6.2cm}+\frac{3}{2} \EE\|\vx^{1}-\vx\|^2+\left(\frac{\alpha}{2\rho}+\frac{1}{2}\right)\EE\|\vz\|^2 +\alpha^2\left(\frac{3}{2}\sigma^2 + 3\beta^2 F^2 G^2+ \frac{3G^2}{M}\frac{C_1}{1-\frac{8\alpha\rho G^2}{M}} + \frac{F^2}{2}\right)\nonumber.
\end{align*}
Since $\vz\ge\vzero$, by the convexity of $f_j$'s and also concavity of $\Psi_\beta$ about $\vz$, we have from the above inequality and the definition of $\phi_1$ in \eqref{eq:def-phi1} that
\begin{align}\label{eq:tmp-ineq1-01}
& \EE\left[f_0(\bar\vx^K)+\frac{1}{M}\sum_{j=1}^M z_j f_j(\bar\vx^K)\right]
\le \EE\left[f_0(\vx)+\Psi_{\beta}(\vx,\bar\vz^K)\right]+\frac{1}{\sqrt{K}}\EE\left[\phi_1(\vx)+\frac{\alpha+\rho}{2\alpha\rho}\|\vz\|^2\right].
\end{align}
Let $\vx=\vx^*$ in the above inequality. Then by Lemma \ref{lem:psi-neg} and the definition of $\Phi$ in \eqref{eq:def-Phi}, we have
$$\EE \big[\Phi(\bar\vx^K; \vx^*, \vz)\big]\le \frac{\phi_1(\vx^*)}{\sqrt{K}}+\frac{1}{\sqrt{K}}\frac{\alpha+\rho}{2\alpha\rho}\EE\|\vz\|^2,\,\forall \vz\ge \vzero.$$
Hence, \eqref{eq:rate-fix-iter-obj} and \eqref{eq:rate-fix-iter-res} follow from the proof of Lemma \ref{lem:pre-rate} and Remark \ref{rm:rate-obj-res}.

Furthermore, as $X$ is bounded, the inequality \eqref{eq:tmp-ineq1-01} implies
\begin{align*}
&~\EE\left[f_0(\bar\vx^K)+\frac{1}{M}\sum_{j=1}^M z_j f_j(\bar\vx^K)\right]\\
\le&~ \EE\left[f_0(\vx)+\Psi_{\beta}(\vx,\bar\vz^K)\right]+\frac{1}{\sqrt{K}}\left[\underset{\vx\in X}\max\phi_1(\vx)+\frac{\alpha+\rho}{2\alpha\rho}\EE\|\vz\|^2\right].
\end{align*}
Therefore, we obtain \eqref{eq:rate-fix-iter-dual} from Lemma \ref{lem:pre-rate} and complete the proof.
\end{proof}

Below we make a few remarks about the results in Theorem \ref{thm:rate-cvx-fix}. Similar remarks also apply to Theorems \ref{thm:rate-cvx-nonfix} and \ref{thm:rate-str-cvx} established later.

\begin{remark}\label{rm:bd-dual}
From the proof of Theorem \ref{thm:rate-cvx-fix}, we see that the setting of $\rho_k$ is for bounding $\EE\|\vz^k\|^2$. If the dual variable $\vz$ is bounded, then $\rho_k$ can be taken as large as the augmented penalty parameter $\beta$. 
\end{remark}

\begin{remark}
By the Markov's inequality $\Prob(\xi\ge \vareps)\le \frac{\EE[\xi]}{\vareps}$ for a nonnegative random variable $\xi$, one can easily have a high-probability result from Theorem \ref{thm:rate-cvx-fix}. One drawback of the result is that in \eqref{eq:rate-fix-iter-res}, the bound is on the average of all inequality constraint violation. Let  
$\gamma=\frac{\EE\left[\max_{j\in [M]}[f_j(\bar\vx^K)]_+\right]}{\EE\left[\frac{1}{M}\sum_{j=1}^M[f_j(\bar{\vx}^{K})]_+\right]}.$ 
Then \eqref{eq:rate-fix-iter-res} implies
$\EE\left[\max_{j\in [M]} [f_j(\bar{\vx}^{K})]_+ \right] \le   \frac{\gamma}{\sqrt{K}}\left(\phi_1(\vx^*) + \frac{\|\vone+\vz^*\|^2}{2\rho}\right).$
If $\gamma = O(1)$, then the maximum violation of the inequality constraint is similar to the avarage violation. However, in the worse case, $\gamma$ could be as large as $M$. 

One may argue that since the averaged constraint violation is used as a measure in the convergence rate result, it could be more natural to work on the equivalent problem \eqref{eq:ccp-csa}, for which only one dual variable is needed instead of the many more $M$ dual variables required in Algorithm \ref{alg:pd-sg}. We point out two potential issues to pursue this direction. First, the augmented Lagrangian function of \eqref{eq:ccp-csa} has a term that is a composition of $\psi_\beta$ given in \eqref{eq:psi} with the finite-sum $\frac{1}{M}\sum_{j=1}^M[f_j(\vx)]_+$. For a stochastic program with such a nested structure, the convergence rate of SGM is much worse \cite{wang2017stochastic} due to the difficulty of obtaining an unbiased SG. Second, the Slater's condition can never hold for \eqref{eq:ccp-csa}. Hence, although one dual variable is needed, the existence of a KKT point is not guaranteed even if the Slater's condition holds for the original problem \eqref{eq:ccp}, and this would affect the convergence analysis. Also, we point out that the use of $M$ dual variables does not cause an issue of memory or computational cost. Compared to the data involved in the $M$ constraint functions, the size of $M$ dual variables is smaller.
\end{remark}

With varying step sizes, we can also show a sublinear convergence rate result of Algorithm \ref{alg:pd-sg} as follows. The order is worse with an additional logarithmic term. 

\begin{theorem}[Convergence rate for convex case with varying step sizes]\label{thm:rate-cvx-nonfix}
Under Assumptions \ref{assump:kkt} through \ref{assump:cvx}, let $\{(\vx^k, \vz^k)\}$ be the sequence generated from Algorithm \ref{alg:pd-sg}. Set parameters according to \eqref{eq:para-nonfix-iter}. For any integer $K\ge 1$, let $\alpha_k=\frac{\alpha}{\sqrt{k+1}\log (k+1)}$ for $1\le k\le K$, $\bar{\vx}^{K}=\frac{1}{\sum_{k=1}^K\alpha_k}\sum_{k=1}^K\alpha_k\vx^{k}$ and $\bar{\vz}^{K}=\frac{1}{\sum_{k=1}^K\alpha_k}\sum_{k=1}^K\alpha_k\vz^{k}$, and define
\begin{equation}\label{eq:def-phi2}
{ \phi_2(\vx)=\frac{3}{2\alpha}\|\vx^{1}-\vx\|^2+2.5\alpha\left(\frac{3}{2}\sigma^2 + 3\beta^2 F^2 G^2 + \frac{3G^2}{M}\frac{C_2}{1-\frac{20\alpha\rho G^2}{M}}+\frac{F^2}{2}\right),}
\end{equation}
with $C_2$ defined in \eqref{eq:def-C2}. Then
{ 
\begin{subequations}\label{eq:rate-nonfix-iter}
\begin{align}
&\EE\big|f_0(\bar{\vx}^{K+1})-f_0(\vx^*)\big|\le\frac{\log(K+1)}{2(\sqrt{K+2}-\sqrt{2})}\left(2\phi_2(\vx^*)+\frac{9(\alpha+\rho)}{2\alpha\rho}\|\vz^*\|^2\right),\\
&\EE\left[\frac{1}{M}\sum_{j=1}^M[f_j(\bar{\vx}^{K+1})]_+\right]\le \frac{\log(K+1)}{2(\sqrt{K+2}-\sqrt{2})}\left(\phi_2(\vx^*)+\frac{\alpha+\rho}{2\alpha\rho}\|\vone+\vz^*\|^2\right).
\end{align}
In addition, if $X$ is bounded, then 
\begin{equation}\label{eq:rate-nonfix-iter-dual}
\EE\big[d_\beta(\vz^*)-d_\beta(\bar\vz^K)\big]\le \frac{3\log(K+1)}{4(\sqrt{K+2}-\sqrt{2})}\left(\max_{\vx\in X}\phi_2(\vx)+\frac{3(\alpha+\rho)}{2\alpha\rho}\|\vz^*\|^2\right).
\end{equation}

\end{subequations}
}
\end{theorem}

\begin{proof}
When the parameters are set according to \eqref{eq:para-nonfix-iter}, we have \eqref{eq:bd-z-nfix}. Hence, multiplying $\alpha_k$ to both sides of \eqref{eq:1iter-bd}, summing it over $k$, and using \eqref{eq:bd-cross-1} through \eqref{eq:bd-cross-3}, we have
\begin{align}\label{eq:tmp-ineq2-0}
\sum_{k=1}^K \alpha_k\EE\left[f_0(\vx^k)+\frac{1}{M}\sum_{j=1}^M z_j f_j(\vx^k)\right]\le &~\sum_{k=1}^K\alpha_k\EE\left[f_0(\vx)+\Psi_{\beta}(\vx,\vz^k)\right]+\frac{3}{2}\EE \|\vx^{1}-\vx\|^2\nonumber\\
&\hspace{-4.5cm}+\left(\frac{\alpha}{2\rho}+\frac{1}{2}\right)\EE\|\vz\|^2+\sum_{k=1}^K\alpha_k^2\left(\frac{3}{2}\sigma^2 + 3\beta^2 F^2 G^2 + \frac{3G^2}{M}\frac{C_2}{1-\frac{20\alpha\rho G^2}{M}}+\frac{F^2}{2}\right).
\end{align}
Note 
\begin{align*}
\sum_{k=1}^K \alpha_k=\sum_{k=1}^K \frac{\alpha}{\sqrt{k+1}\log(k+1)}
\ge \frac{\alpha}{\log(K+1)}\int_{1}^{K+1}\frac{1}{\sqrt{x+1}}dx
=\frac{2\alpha(\sqrt{K+2}-\sqrt{2})}{\log(K+1)}.
\end{align*}
Hence, dividing both sides of \eqref{eq:tmp-ineq2-0} by $\sum_{k=1}^K\alpha_k$, we have from the convexity of $f_j$'s and the concavity of $\Psi_\beta$ about $\vz$, and also using $\sum_{k=1}^K\alpha_k^2\le 2.5$ from \eqref{eq:est-sum-logk} and the definition of $\phi_2$ in \eqref{eq:def-phi2} that
\begin{align*}
&~\textstyle\EE\left[f_0(\bar\vx^K)+\frac{1}{M}\sum_{j=1}^M z_j f_j(\bar\vx^K)\right]\cr
\le &~\textstyle\EE\left[f_0(\vx)+\Psi_{\beta}(\vx,\bar\vz^K)\right]+\frac{\log(K+1)}{2(\sqrt{K+2}-\sqrt{2})}\EE\left[\phi_2(\vx)+\frac{\alpha+\rho}{2\alpha\rho}\|\vz\|^2\right].
\end{align*}
Now following the same arguments as those below \eqref{eq:tmp-ineq1-01} in the proof of Theorem \ref{thm:rate-cvx-fix}, we obtain the desired results and complete the proof.
%
\end{proof}

\subsection{Convergence rate for strongly convex problems}
In this subsection, we analyze the convergence rate of Algorithm \ref{alg:pd-sg} for strongly convex problems, i.e., $\mu>0$ in \eqref{eq:str-cvx-ineq}. Similar to the convex case, we first bound $\EE\|\vz^k\|^2$ by choosing appropriate parameters. The proof is shown in the appendix.
\begin{proposition}\label{prop:str-cvx}
Under Assumptions \ref{assump:kkt} through \ref{assump:cvx} with $\mu>0$, for any given positive integer $K$, let $\{(\vx^k,\vz^k)\}$ be the sequence generated from Algorithm \ref{alg:pd-sg} with parameters set to
\begin{equation}\label{eq:para-str-cvx}
\vD_k=\frac{k+1}{\alpha} \vI,\ \rho_k = 
\frac{\rho}{\log (K+1)},\ \beta\ge \frac{2\rho}{\log 2},\,\forall 1\le k\le K,
\end{equation}
where $\alpha \ge \frac{1}{\mu}$ and $\alpha\rho < \frac{M}{8 G^2}$. Then for any $1\le k\le K+1$,
\begin{equation}\label{eq:bd-z-str-cvx}
\EE\|\vz^k\|^2 \le \frac{C_3}{1-\frac{8\alpha\rho G^2}{M}},
\end{equation}
where
\begin{equation}\label{eq:def-C3}
C_3=\frac{2\rho}{\log(K+1)}\left(\frac{2}{\alpha}-\mu\right)\|\vx^1-\vx^*\|^2+4\|\vz^*\|^2+4\alpha\rho(\sigma^2+2\beta^2F^2G^2).
\end{equation}
\end{proposition}

{ Similar to Lemma \ref{lem:bd-cross}, we have the following result bounding the expectation terms in \eqref{eq:1iter-bd}. The proof is also given in the appendix.
\begin{lemma}\label{lem:bd-cross-str}
Under the assumptions of Proposition \ref{prop:str-cvx}, for any deterministic or stochastic vector $\vz\ge\vzero$,  we have
\begin{align}\label{eq:bd-cross-str-3}
&~\sum_{k=1}^K \EE\left[\big\langle\vz^k-\vz, M\ve_{j_k}\odot \nabla_\vz \Psi(\vx^k,\vz^k) - \nabla_\vz \Psi(\vx^k,\vz^k) \big\rangle\right] \cr
\le &~\frac{\log(K+1)}{2\rho}\EE\left[\|\vz^1-\vz\|^2 + \sum_{k=1}^K \|\vz^{k+1}-\vz^k\|^2\right]. 
\end{align}
\end{lemma}
}

Using \eqref{eq:1iter-bd} and \eqref{eq:bd-z-str-cvx}, we establish the convergence rate result of Algorithm \ref{alg:pd-sg} for the case of $\mu>0$ as follows.
\begin{theorem}[convergence rate for strongly convex case]\label{thm:rate-str-cvx}
Under the assumptions of Proposition \ref{prop:str-cvx}, we have
\begin{equation}\label{eq:rate-nonerg}
\EE\|\vx^{K+1}-\vx^*\|^2\le \frac{2\alpha}{K+1}\left(\phi_3(\vx^*)+\frac{\log(K+1)}{\rho}\|\vz^*\|^2\right),
\end{equation}
where
\begin{equation}\label{eq:def-phi3}
\phi_3(\vx)=\left(\frac{1}{\alpha}-\frac{\mu}{2}\right)\|\vx^{1}-\vx\|^2+\alpha\log(K+1)\left(\sigma^2 + 2\beta^2 F^2 G^2 + \frac{2G^2}{M}\frac{C_3}{1-\frac{8\alpha\rho G^2}{M}}\right),
\end{equation}
with $C_3$ defined in \eqref{eq:def-C3}.
In addition, let $\bar{\vx}^{K}=\frac{\sum_{k=1}^K\vx^{k}}{K}$ and $\bar{\vz}^{K}=\frac{\sum_{k=1}^K\vz^{k}}{K}$. Then
\begin{subequations}\label{eq:rate-str-cvx-obj-res}
\begin{align}
&\EE\big|f_0(\bar{\vx}^{K})-f_0(\vx^*)\big| \le \frac{1}{K}\left(2\phi_3(\vx^*) + \frac{9\log(K+1)}{\rho}\|\vz^*\|^2\right),\\
& \EE\left[\frac{1}{M}\sum_{j=1}^M[f_j(\bar{\vx}^{K})]_+\right]\le\frac{1}{K}\left(\phi_3(\vx^*) + \frac{\log(K+1)}{\rho}\|\vone+\vz^*\|^2\right).
\end{align}
\end{subequations}
\end{theorem}

\begin{proof}
Let $\alpha_k = \frac{\alpha}{k+1},\,\forall\, k\ge 1$. Since $\alpha\ge\frac{1}{\mu}$, it holds $\frac{k+1}{\alpha}\ge\frac{k+2}{\alpha}-\mu$, i.e., $\frac{1}{\alpha_k}\ge \frac{1}{\alpha_{k+1}}-\mu$. Hence, summing up \eqref{eq:1iter-bd} with $\vx=\vx^*$ from $k=1$ through $K$, using Lemma \ref{lem:bd-cross-str}, and noting $\vz^1=\vzero$ and the choice of $\rho_k$ yield
\begin{align}\label{eq:str-cvx-bd-z1-0}
&~\sum_{k=1}^K\EE\left[f_0(\vx^k)+\frac{1}{M}\sum_{j=1}^M z_j f_j(\vx^k)\right]+\frac{1}{2\alpha_K}\EE\|\vx^{K+1}-\vx^*\|^2+\frac{1}{2\rho_K}\EE\|\vz^{K+1}-\vz\|^2\nonumber\\
\le &~\sum_{k=1}^K\EE\left[f_0(\vx^*)+\Psi_{\beta}(\vx^*,\vz^k)\right]+\left(\frac{1}{2\alpha_1}-\frac{\mu}{2}\right)\|\vx^{1}-\vx^*\|^2+\frac{1}{\rho_1}\EE\|\vz\|^2\nonumber\\
&~+\sum_{k=1}^K\alpha_k\left(\sigma^2 + 2\beta^2 F^2 G^2 + \frac{2G^2}{M}\EE\|\vz^k\|^2\right){ -\sum_{k=1}^K\frac{1}{2\rho_k}\left(\frac{\beta }{\rho_k}-2\right)\EE\|\vz^{k+1}-\vz^k\|^2}\nonumber\\
\le & ~Kf_0(\vx^*)+\EE\left[\phi_3(\vx^*)+\frac{1}{\rho_1}\|\vz\|^2\right],
\end{align}
where in the first inequality, { we have used the fact $\EE\big\langle \vx^{k}-\vx^*, \vg_0^k-\tilde{\nabla}f_0(\vx^k)\big\rangle=0$ and $\EE\big\langle \vx^{k}-\vx^*, \vh^k-\nabla_\vx\Psi_\beta(\vx^k,\vz^k)\big\rangle=0$}, and in the second inequality, we have used \eqref{eq:bd-z-str-cvx} and \eqref{eq:est-k-logk}, Lemma \ref{lem:psi-neg}, the setting $\beta\ge 2\rho_k, \,\forall\,k$, and also the definition of $\phi_3$ in \eqref{eq:def-phi3}.
Let $\vz=\vz^*$ in the above inequality. Then by \eqref{eq:opt-cond}, we have that
\begin{align*}
\frac{1}{2\alpha_K}\EE\|\vx^{K+1}-\vx^*\|^2
\le \phi_3(\vx^*)+\frac{1}{\rho_1}\|\vz^*\|^2,
\end{align*}
which clearly implies \eqref{eq:rate-nonerg} by the parameters given in \eqref{eq:para-str-cvx} and also $\alpha_K=\frac{\alpha}{K+1}$. 

Furthermore, dropping the terms about $\|\vx^{K+1}-\vx^*\|^2$ and $\|\vz^{K+1}-\vz\|^2$ on the left hand side of \eqref{eq:str-cvx-bd-z1-0}, and using the convexity of $f_j$'s, we have for any $\vz\ge\vzero$ that
\begin{align*}
\textstyle\EE\left[f_0(\bar\vx^K)-f_0(\vx^*)+\frac{1}{M}\sum_{j=1}^M z_j f_j(\bar\vx^K)\right]\le  \frac{1}{K}\EE\left[\phi_3(\vx^*)+\frac{1}{\rho_1}\|\vz\|^2\right].
\end{align*}
Now using Lemma \ref{lem:pre-rate} and Remark \ref{rm:rate-obj-res}, we obtain the desired results.
\end{proof}

\begin{remark}\label{rm:scvx-rate}
The order of the established rate is worse than the optimal one obtained for a primal SGM by a $\log(K+1)$ factor. That term appears essentially because of the setting of $\rho_k$ to bound the dual iterate. If we assume $\{\vz^k\}$ to be bounded, then we can set $\rho_k=\frac{\beta}{2}$ and remove the logarithmic term.
Furthermore, if the maximum number $K$ of iteration is not given, we can set $$\vD_k = \frac{k+1}{\alpha}\vI,\, \rho_k=\frac{\rho}{\log(k+1)},\, \beta\ge\frac{2\rho}{\log 2},\,\forall k,$$
with $\alpha\ge \frac{1}{\mu}$. These parameters satisfy the conditions in Proposition \ref{prop:bd-z-prelim}, and thus we can still have a sublinear convergence result through first bounding $\EE\|\vz^k\|^2$. However, there will be an additional $\log(K+1)$ term in the obtained result, i.e., $O\big([\log(K+1)]^2/(K+1)\big)$ for any positive integer $K$. The result can be shown by following the proofs of Proposition \ref{prop:str-cvx} and Theorem \ref{thm:rate-str-cvx}. We leave it to the interested readers. 
\end{remark}

{ 
\section{Convergence analysis of the adaptive method}\label{sec:analysis-adp}
In this section, we analyze Algorithm \ref{alg:pd-sg} with the adaptive Setting \ref{set:adp} for $\vD_k$'s. For simplicity and also due to the page limitation, we only consider the convex case with pre-determined maximum number of iterations. For the convex case with varying maximum number of iterations and the strongly convex case, we can have similar results as those in section \ref{sec:analysis}. 

Similar to the analysis in the previous section, we first bound $\EE\|\vz^k\|^2$ as follows. Its proof is given in the appendix.

\begin{proposition}\label{prop:adp}
Assume that $X$ is bounded and also Assumptions \ref{assump:kkt} through \ref{assump:cvx} hold. Given a positive integer $K$, let $\alpha>0$ and $\rho>0$ such that $\alpha\rho < \frac{M}{8G^2}$, and let
\begin{equation}\label{eq:para-adp}
\alpha_k= \frac{\alpha}{\sqrt{K}},\, \rho_k = \frac{\rho}{\sqrt{K}},\, \beta \ge \rho, \,\forall 1\le k\le K.
 \end{equation}
Suppose that $\{(\vx^k,\vz^k)\}$ is generated from Algorithm \ref{alg:pd-sg} with $\vD_k$ set according to Setting \ref{set:adp} and all other parameters specified in \eqref{eq:para-adp}. Then for any $\vx\in X$, we have
\begin{equation}\label{eq:bd-sum_Dk-3}
\frac{1}{2}\sum_{k=1}^K\big(\|\vx^k-\vx\|_{\vD_k}^2-\|\vx^{k+1}-\vx\|_{\vD_k}^2\big)\le \frac{\sqrt K}{2\alpha}\|\vx^1-\vx\|^2 + \frac{\eta B^2\sqrt{nK}}{2}.
\end{equation}
In addition, for any $1\le k\le K+1$, it holds that
\begin{equation}\label{eq:bd-z-adp}
\EE\|\vz^k\|^2 \le \frac{C_4}{1- \frac{8\alpha\rho G^2}{M}}.
\end{equation}
Here $B = \max_{\vx_1,\vx_2\in X}\|\vx_1-\vx_2\|_\infty,$ and
\begin{equation}\label{eq:def-C4}
C_4= \frac{2\rho}{\alpha}\|\vx^1-\vx^*\|^2 + 2\rho \eta B^2\sqrt{n}+4\|\vz^*\|^2 + 4\alpha\rho\big(\sigma^2 + 2\beta ^2 F^2 G^2\big).
\end{equation}
\end{proposition}

By the above proposition, we have the convergence rate estimate of Algorithm \ref{alg:pd-sg} with the adaptive Setting \ref{set:adp} about $\vD_k$'s.
\begin{theorem}\label{thm:adp}
Under Assumptions \ref{assump:kkt} through \ref{assump:cvx}, let $\{(\vx^k,\vz^k)\}$ be generated from Algorithm \ref{alg:pd-sg} with $\vD_k$ set according to Setting \ref{set:adp}. Given any positive integer $K$, set the parameters according to \eqref{eq:para-adp}, let $\bar{\vx}^{K}=\frac{1}{K}\sum_{k=1}^K\vx^{k}$ and $\bar\vz^K=\frac{1}{K}\sum_{k=1}^K\vz^{k}$, and define
\begin{equation}\label{eq:def-phi4}
\phi_4(\vx)=\frac{3}{2\alpha}\|\vx^{1}-\vx\|^2+\frac{\eta B^2\sqrt{n}}{2}+\alpha\left(\frac{3}{2}\sigma^2 + 3\beta^2 F^2 G^2 + \frac{3G^2}{M}\frac{C_4}{1-\frac{8\alpha\rho G^2}{M}}+\frac{F^2}{2}\right),
\end{equation}
where $C_4$ is defined in \eqref{eq:def-C4}.
If $X$ is bounded, then
\begin{subequations}\label{eq:rate-adp}
\begin{align}
&\EE\big|f_0(\bar{\vx}^{K})-f_0(\vx^*)\big| \le \frac{1}{\sqrt{K}} \left(2\phi_4(\vx^*)+ \frac{9(\alpha+\rho)}{2\alpha\rho}\|\vz^*\|^2\right),\label{eq:rate-adp-obj}\\
&\EE\left[\frac{1}{M}\sum_{j=1}^M[f_j(\bar{\vx}^{K})]_+\right] \le   \frac{1}{\sqrt{K}}\left(\phi_4(\vx^*) + \frac{\alpha+\rho}{2\alpha\rho}\|\vone+\vz^*\|^2\right).\label{eq:rate-adp-res}\\
&\EE\big[d_\beta(\vz^*)-d_\beta(\bar\vz^K)\big]\le \frac{3}{2\sqrt{K}}\left(\max_{\vx\in X}\phi_4(\vx)+\frac{3(\alpha+\rho)}{2\alpha\rho}\|\vz^*\|^2\right).\label{eq:rate-adp-dual}
\end{align}
\end{subequations}
\end{theorem}

\begin{proof}
Multiply $\alpha_k=\frac{\alpha}{\sqrt K}$ to \eqref{eq:1iter-bd}, sum it up from $k=1$ through $K$, use \eqref{eq:bd-cross-1} through \eqref{eq:bd-cross-3} and also \eqref{eq:bd-sum_Dk-3}, and note $\vz^1=\vzero$. Then we have from \eqref{eq:bd-z-adp} that 
\begin{align*}
\frac{\alpha}{\sqrt{K}}\sum_{k=1}^K \EE\left[f_0(\vx^k)+\frac{1}{M}\sum_{j=1}^M z_j f_j(\vx^k)\right]\le &~\frac{\alpha}{\sqrt{K}}\sum_{k=1}^K\EE\left[f_0(\vx^*)+\Psi_{\beta}(\vx,\vz^k)\right]+\frac{\alpha \eta B^2\sqrt{n}}{2}\nonumber\\
&\hspace{-6.2cm}+\frac{3}{2} \EE\|\vx^{1}-\vx\|^2+\left(\frac{\alpha}{2\rho}+\frac{1}{2}\right)\EE\|\vz\|^2 +\alpha^2\left(\frac{3}{2}\sigma^2 + 3\beta^2 F^2 G^2+ \frac{3G^2}{M}\frac{C_4}{1-\frac{8\alpha\rho G^2}{M}} + \frac{F^2}{2}\right)\nonumber.
\end{align*}
Now the desired results can be obtained by following the same arguments as those in the proof of Theorem \ref{thm:rate-cvx-fix}.
\end{proof}

\begin{remark}
From the proofs of Proposition \ref{prop:adp} and Theorem \ref{thm:adp}, we see that the inequality \eqref{eq:bd-sum_Dk-3} is important to bound $\EE\|\vz^k\|^2$ and to have the convergence rate results. In addition, while proving \eqref{eq:bd-sum_Dk-3}, we use the bound $\|\vs^k\|=O(\sqrt{k})$. Since we scale the SGs in Setting \ref{set:adp}, we automatically have such a bound. Without the scaling process, we may not have it unless we assume the dual variable to be bounded.
\end{remark}
}

\section{Numerical experiments}\label{sec:numerical}

In this section, we test the proposed method (named PDSG) on solving a sample approximation problem of the robust portfolio selection (RPS) and also three quadratically constrained quadratic programs (QCQP). We compare to the stochastic mirror-prox method in \cite{juditsky2011solving} and the CSA method in \cite{lan2016algorithms-exp-cont}. { The RPS test is performed in MATLAB 2016a installed on a Macbook Pro with 8 gigabyte memory, while the QCQP test is in MATLAB 2018a installed on a Dell workstation with 32 gigabyte memory.}

\subsection{Sample approximation of robust portfolio selection}
Suppose that one investor has a unit of capital to invest on $n$ assets. Assume the return rate of the $i$-th asset follows a uniform distribution on $[\mu_i-\sigma_i, \mu_i+\sigma_i]$ for each $i\in [n]$. The RPS aims to maximize the expected return subject to a minimum return $c$ for all possible return rate, i.e., 
\begin{equation}\label{eq:rps}
\max_{\vx\in X} \bm{\mu}^\top \vx, \st  \sum_{i=1}^n \xi_i x_i \ge c, \forall \xi_i\in [\mu_i-\sigma_i, \mu_i+\sigma_i], \forall i\in [n], 
\end{equation}  
where $X=\{\vx: \vx\ge\vzero, \sum_{i=1}^n x_i = 1\big\}$.
It is easy to see that the above robust constraint is equivalent to $\sum_{i=1}^n (\mu_i-\sigma_i)x_i\ge c$, and thus \eqref{eq:rps} can be equivalently formulated as a linear program with only two linear constraints and also the nonnegativity constraint. 

Now suppose that the distribution of the return rate $\bm{\xi}$ is unknown but its samples are available. Let $\{\bm{\xi}_1,\ldots,\bm{\xi}_M\}$ be $M$ samples of $\bm{\xi}$ and $\bar{\bm{\mu}}$ be the empirical mean. Then we can solve a sample approximation of \eqref{eq:rps}, i.e.,
\begin{equation}\label{eq:rps-app}
\max_{\vx\in X} \bar{\bm{\mu}}^\top \vx, \st  \bm{\xi}_j^\top \vx \ge c, \forall j\in [M]. 
\end{equation} 
The sample approximation problem is still a linear program, and one can apply any linear program solver. We use the proposed method in this test simply to see if it can numerically perform well. We set $n=10$ and $M=10^4$. All entries of $\bar{\bm{\mu}}$ are generated independently following the uniform distribution on $[1,2]$. For each $j\in [M]$, we set $\bm{\xi}_j=\bar{\bm{\mu}}+\bm{\zeta}_j$ with $\bm{\zeta}_j$ generated by uniform distribution on $[-0.5,0.5]^n$. Then we let 
$c=0.9\min_{j\in[M], \vx\in X}\bm{\xi}_j^\top \vx$
to ensure that \eqref{eq:rps-app} has a strict feasible solution. The parameters of our algorithm are set according to \eqref{eq:para-fix-iter} with $K=10^6$, and $\alpha=\rho=\beta=1$. The initial point is randomly generated. Figure \ref{fig:rspf} shows the distance of objective value to the optimal value, the averaged constraint violation, and also the maximum constraint violation, where the optimal objective value is obtained by MATLAB's built-in function \verb|linprog|. The feasibility curves only show the first 5,000 iterations, after which the points remain feasible. { We also test the mirror-prox method \cite{juditsky2011solving} and the CSA method \cite{lan2016algorithms-exp-cont} and find that they perform almost the same as our method on this simple example.}

\begin{figure}[h]
\begin{center}
\includegraphics[width=0.325\textwidth]{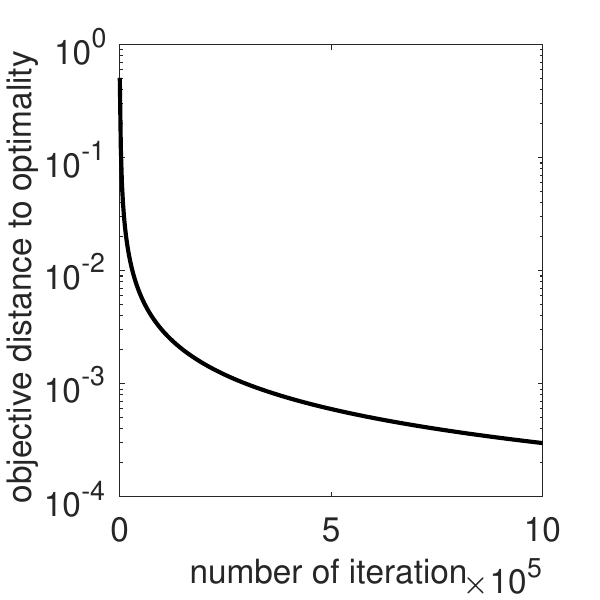}
\includegraphics[width=0.325\textwidth]{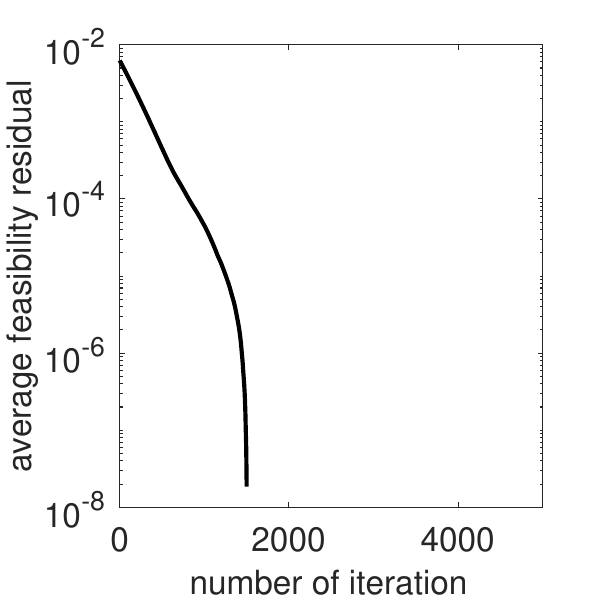}
\includegraphics[width=0.325\textwidth]{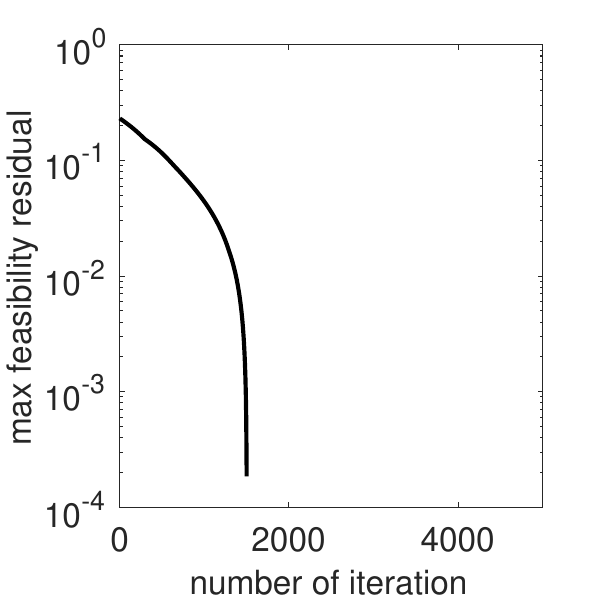}
\end{center}
\caption{Results given by Algorithm \ref{alg:pd-sg} with nonadaptive setting on solving an instance of the sample approximation \eqref{eq:rps-app} of the robust portfolio selection. Left: the distance of objective value at averaged point to optimal value $|f_0(\bar\vx^k)-f_0(\vx^*)|$; Middle: the average constraint violation at averaged point $\frac{1}{M}\sum_{j=1}^M [f_j(\bar\vx^k)]_+$; Right: the maximum constraint violation at averaged point $\max_{j\in [M]} [f_j(\bar\vx^k)]_+$.}\label{fig:rspf}
\end{figure}

\subsection{Quadratically constrained quadratic program}
In this subsection, we test the proposed method on a finite-sum structured quadratic program with many quadratic constraints, i.e., 
\begin{equation}\label{eq:cvx-qcqp}
\min_{\vx\in X} \frac{1}{2N}\sum_{i=1}^N\|\vH_i \vx - \vc_i\|^2, \st \frac{1}{2}\vx^\top \vQ_j\vx + \va_j^\top \vx \le b_j, j=1,\ldots, M.
\end{equation}
Here $X=[-10,10]^n$; for each $i\in [N]$, $\vH_i\in\RR^{p\times n}$ and $\vc_i$ are randomly generated with components independently following standard Gaussian distribution; the entries of every $\va_j$ also follow standard Gaussian distribution; $\vQ_j$'s are randomly generated symmetric positive semidefinite matrices; each $b_j$ is generated according to uniform distribution on $[0.1,1.1]$. Note that for the generated data, the Slater's condition holds, and thus there must exist a KKT point for \eqref{eq:cvx-qcqp}. Let $\xi$ be a random variable with uniform distribution on $[N]$. Then the objective of  \eqref{eq:cvx-qcqp} can be written to $\EE_\xi \frac{1}{2}\|\vH_\xi\vx-\vc_\xi\|^2$, and thus \eqref{eq:cvx-qcqp} is in the form of \eqref{eq:ccp}. 

In the experiment, { we test on three QCQP instances of different size. For all of them, we set $N=M=10^4$ in \eqref{eq:cvx-qcqp}, and the dimension $(n,p)$ is set to $(10,5)$, $(200, 150)$, and $(400,350)$ respectively for the three instances. 
We test the proposed algorithm with both nonadaptive and adaptive settings.} For the nonadaptive one, we set algorithm parameters according to \eqref{eq:para-fix-iter} with $K=50,000$, $\alpha=\rho=\sqrt{10}$, and $\beta=1$, and it is named as \verb|PDSG-nonadp|. 
{ For the adaptive method, i.e., $\vD_k$ given according to Setting \ref{set:adp}, we set $\eta = \frac{1}{\sqrt{10}}$ and the other parameters according to \eqref{eq:para-adp} with $K=50,000$, $\alpha = 10$, $\rho = \sqrt{10}$, and $\beta=1$, and we name it as \verb|PDSG-adp|.} The stochastic mirror-prox method \cite{juditsky2011solving} with update given in \eqref{eq:sd-update} is applied on the equivalent saddle-point problem \eqref{eq:sdl-prob}. Although the mirror-prox method requires a compact $Z$, we simply set $Z=\RR^M$, and the method still works well in this test. We use the same penalty parameter $\beta=1$ and the same step size $\alpha_k$ as for our nonadaptive method. Also we apply the CSA method  \cite{lan2016algorithms-exp-cont} with update given in \eqref{eq:csa-update}. The same step size $\alpha_k$ is used, and $\eta_k$ is set to $1/\sqrt{K}$ for all $k$. For all the tested methods, at each iteration, we sample 10 component functions in the objective and also 10 constraint functions to obtain an unbiased SG, i.e., mini-batch of size 10 is applied. Projecting onto the set $\{\vx: \frac{1}{2}\vx^\top \vQ\vx + \va^\top \vx\le b\}$ does not generally admit an analytic solution and requires an iterative method. Hence, the methods in \cite{wang2016stochastic-proj, ryu2017ppg} with updates \eqref{eq:alg-proj} and \eqref{eq:sppg-update} could be inefficient on solving the QCQP problem and are not compared.  

Figure \ref{fig:qcqp} shows the results for each method on the three QCQP instances, including the objective error, average constraint violation, and also maximum constraint violation with respect to epoch, where the ``optimal'' solution is computed by running \verb|PDSG-adp| to 1,000 epochs for the smallest instance and 500 epochs for another two.  
{ Table \ref{table:time} shows the running time (in second) of each method. Since all the tested methods have almost the same per-iteration complexity, their total running times are almost the same. The very long time for the largest instance is because the data size in this instance almost reaches the limit of machine memory.} From the results, we see that the proposed algorithm performs significantly better than the stochastic mirror-prox and CSA methods. In addition, { the adaptive PDSG is significantly better than the nonadaptive one. Note that we scale the SGs in the adaptive PDSG. Hence, with the parameters we set, the two PDSGs use roughly the same step size in this experiment. Therefore, the better performance of the adaptive method is mainly attributed to its different setting of $\vD_k$.} 

\begin{figure}[h]
\begin{center}
\includegraphics[width=0.325\textwidth]{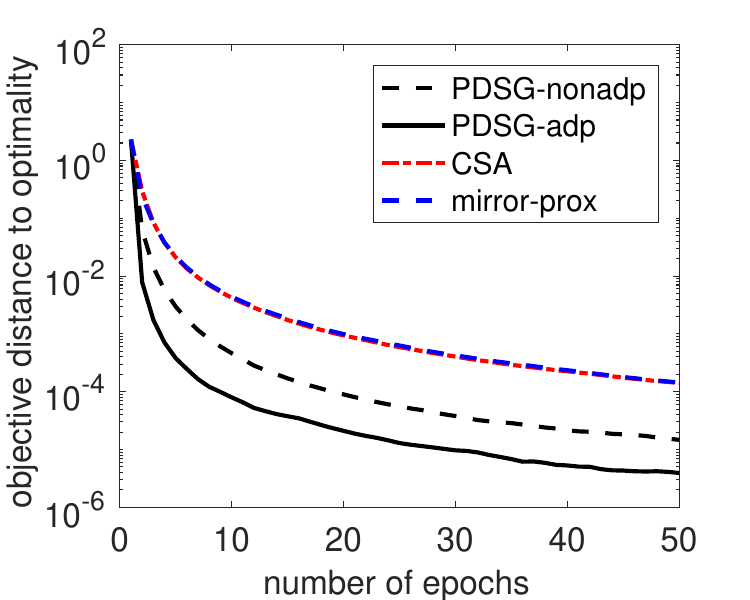}
\includegraphics[width=0.325\textwidth]{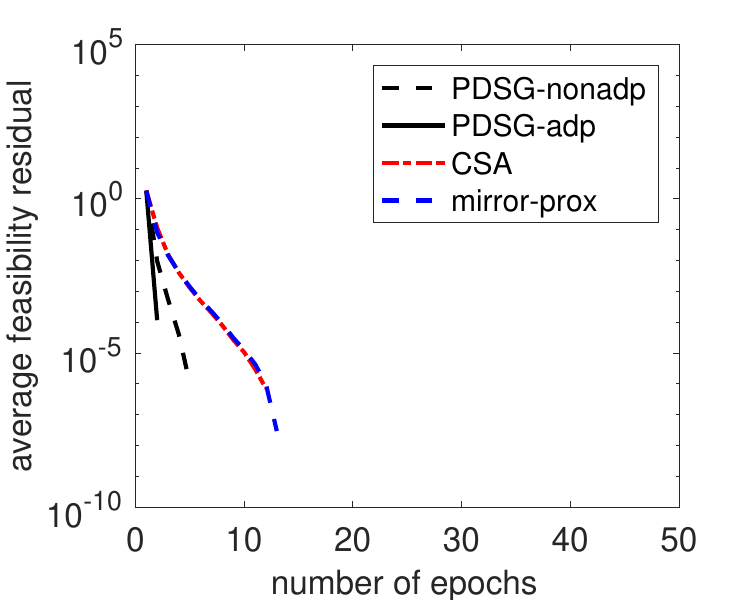}
\includegraphics[width=0.325\textwidth]{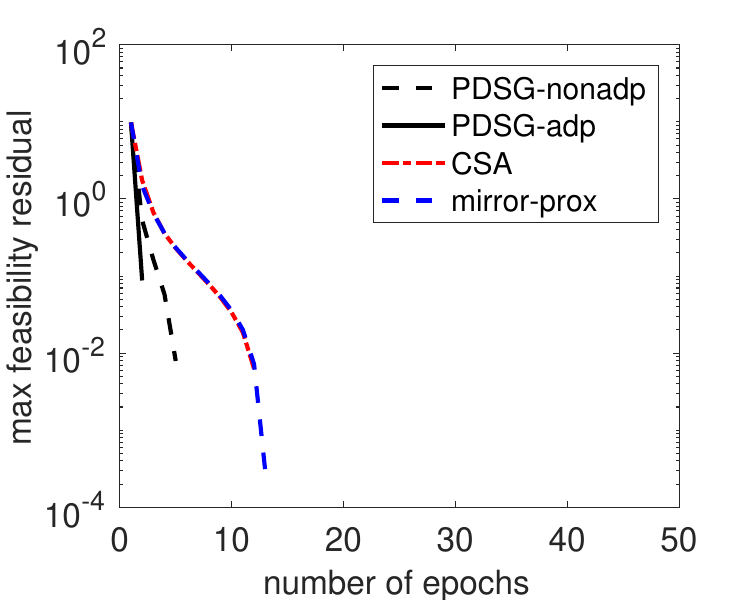}\\[0.1cm]
\includegraphics[width=0.325\textwidth]{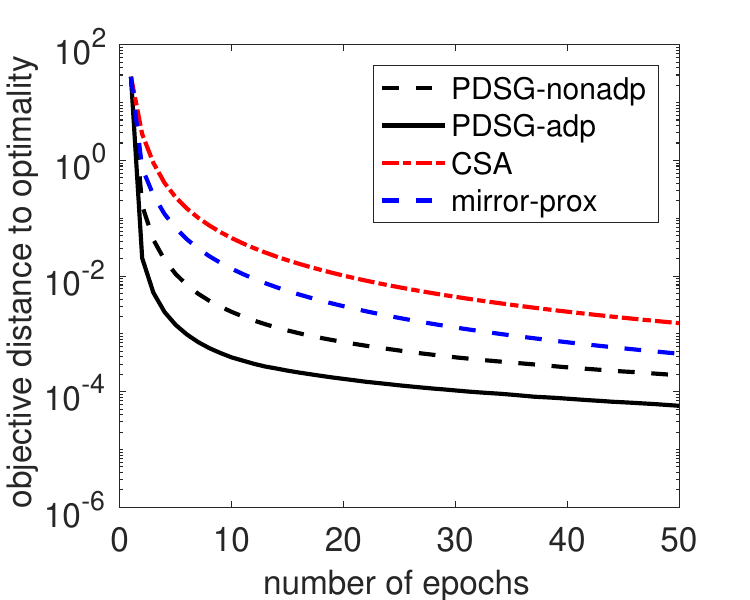}
\includegraphics[width=0.325\textwidth]{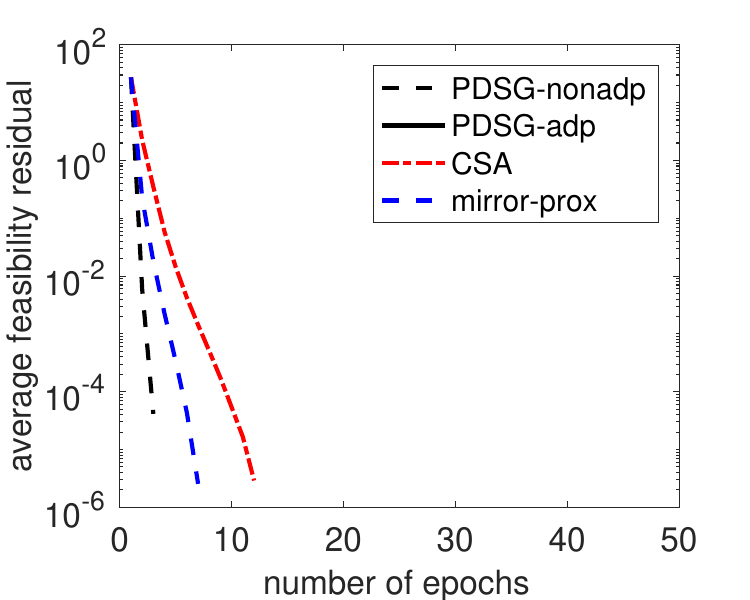}
\includegraphics[width=0.325\textwidth]{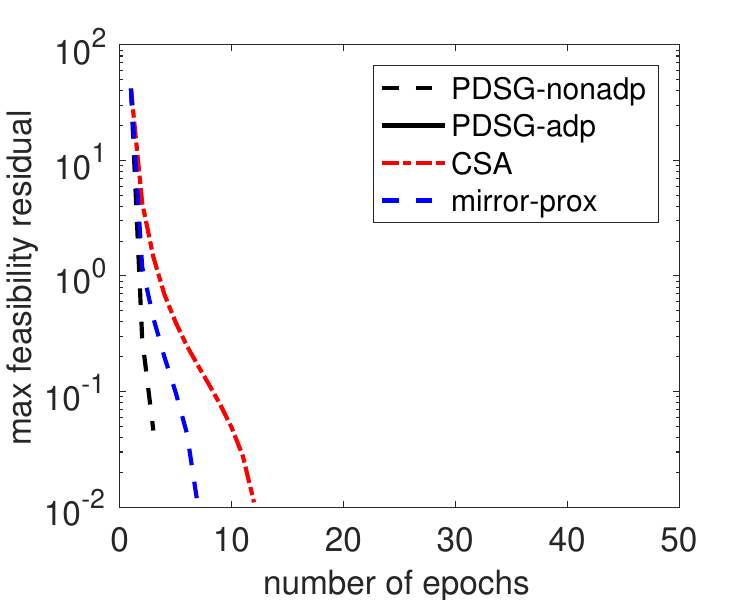}\\
[0.1cm]
\includegraphics[width=0.325\textwidth]{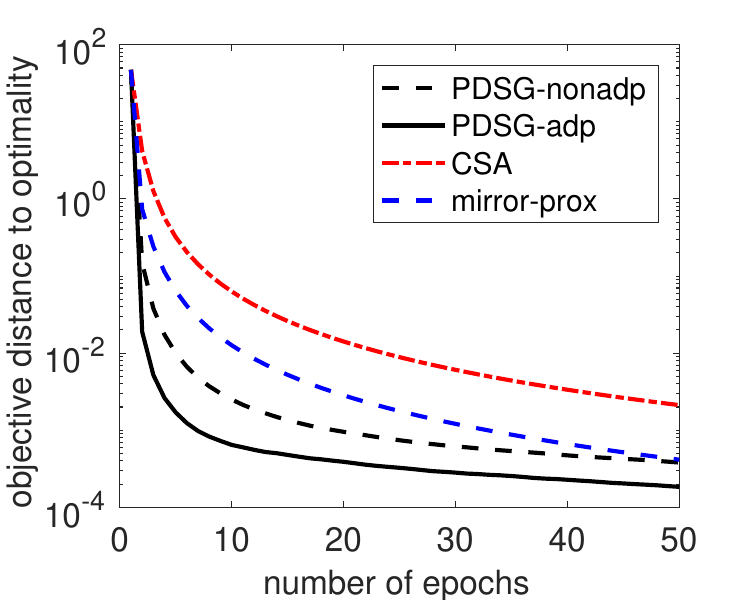}
\includegraphics[width=0.325\textwidth]{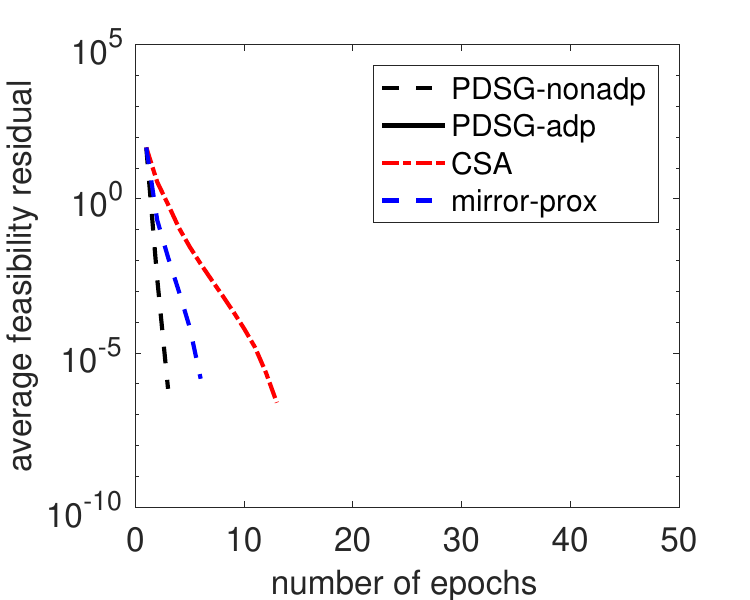}
\includegraphics[width=0.325\textwidth]{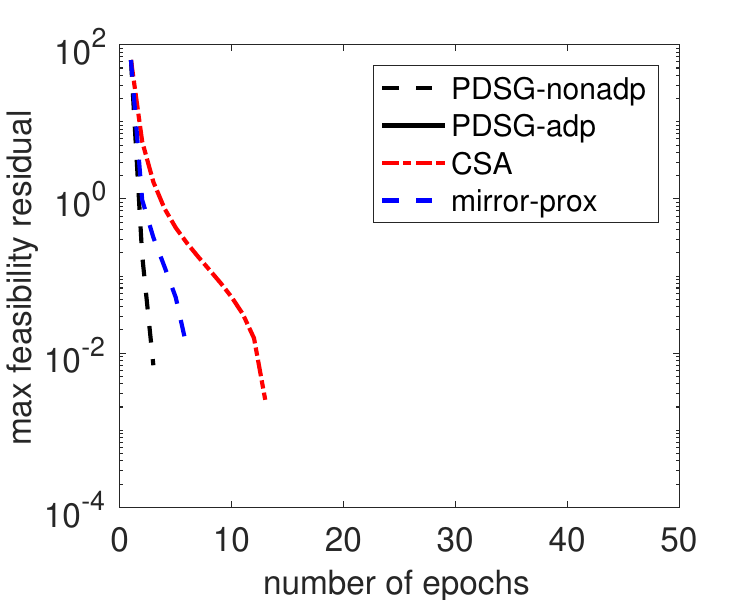}\\
\end{center}
\caption{Results given by Algorithm \ref{alg:pd-sg} with both nonadaptive and adaptive settings (named PDSG-nonadp and PDSG-adp), the stochastic mirror-prox method in \cite{juditsky2011solving}, and the CSA method in \cite{lan2016algorithms-exp-cont} on solving three instances of the quadratically constrained quadratic programming \eqref{eq:cvx-qcqp}, each instance with $N=M=10,000$. Left: the distance of objective value at averaged point to optimal value $|f_0(\bar\vx^k)-f_0(\vx^*)|$; Middle: the average constraint violation at averaged point $\frac{1}{M}\sum_{j=1}^M [f_j(\bar\vx^k)]_+$; Right: the maximum constraint violation at averaged point $\max_{j\in [M]} [f_j(\bar\vx^k)]_+$. First row: dimension $n=10, p=5$; Second row: dimension $n=200,p=150$; Last row: dimension $n=400,p=350$.}\label{fig:qcqp}
\end{figure}

\begin{table}[h]
\begin{center}
{\small\begin{tabular}{|c||cccc|}
\hline
\backslashbox{Dimension}{Method}& PDSG-nonadp & PDSG-adp & CSA & mirror-prox \\\hline
$n=10, p=5$ & 20.20  &20.69  &20.56  &20.32\\\hline\hline
$n=200,p=150$ & 248.94  &239.56  &250.01  &244.82\\\hline\hline
$n=400,p=350$ & 20129.59  &20044.41  &20118.03  &20161.85\\\hline
\end{tabular}
}
\end{center}
\caption{Running time (in second) for each compared method on three instances tested in Figure \ref{fig:qcqp}.}\label{table:time}
\end{table}

\section{Conclusions}\label{sec:conclusion}
We have proposed a primal-dual { (adaptive)} stochastic gradient method for stochastic programming with many functional constraints. Every iteration, the method only needs a stochastic subgradient of the objective, and a subgradient and the function value of one randomly sampled constraint function. 
Under standard assumptions, we have established its convergence rate for both convex and strongly convex problems. The order of rate is optimal for convex case and nearly optimal for strongly convex case. Numerical experiments on a sample approximation problem of the robust portfolio selection and  
quadratically constrained quadratic programming demonstrate its nice practical performance.

{ 
\section*{Acknowledgements} The author would like to thank the two anonymous referees for their constructive comments and suggestions, which greatly improve the paper. In particular, he very much appreciates the careful checking from one of them, who pointed out one technical mistake in the first submission. The author also would like to thank Professor Wotao Yin for his valuable discussions.
}

\appendix
{\small
\section{Proofs of Propositions}

\subsection{Proof of Proposition \ref{prop:bd-z-prelim}}
Let $(\vx,\vz)=(\vx^*,\vz^*)$ in \eqref{eq:1iter-bd}. {\color{blue}Then the last three expectation terms vanish.} Since $\rho_k\le \beta$, we have by the definition of $\Phi$ in \eqref{eq:def-Phi} and Lemma \ref{lem:psi-neg} that
\begin{align*}
&~\textstyle\EE\left[\Phi(\vx^k;\vx^*,\vz^*)\right]+\frac{1}{2}\EE\|\vx^{k+1}-\vx^*\|_{\vD_k}^2+\frac{1}{2\rho_k}\EE\|\vz^{k+1}-\vz^*\|^2\nonumber\\
\le &~\textstyle\frac{1}{2}\EE \|\vx^{k}-\vx^*\|_{\vD_k-\mu\vI}^2+\frac{1}{2\rho_k}\EE\|\vz^k-\vz^*\|^2+\alpha_k\left(\sigma^2 + 2\beta ^2 F^2 G^2 + \frac{2G^2}{M}\EE\|\vz^k\|^2\right).\nonumber
\end{align*}
Multiplying $2\rho_k$ to both sides of the above inequality gives
\begin{align*}
&~2\rho_k\EE\big[\Phi(\vx^{k}; \vx^*,\vz^*)\big]+\rho_k\EE\|\vx^{k+1}-\vx^*\|_{\vD_k}^2+\EE\|\vz^{k+1}-\vz^*\|^2\\
\le &~\rho_k\EE \|\vx^{k}-\vx^*\|_{\vD_k-\mu\vI}^2+\EE\|\vz^k-\vz^*\|^2+2\alpha_k\rho_k\left(\sigma^2 + 2\beta ^2 F^2 G^2 + \frac{2G^2}{M}\EE\|\vz^k\|^2\right).\nonumber
\end{align*}
Summing the above inequality from $k=1$ through $t$, we have by $\vz^1=\vzero$, noting $\Phi(\vx^{k}; \vx^*,\vz^*)\ge0,\,\forall k$ from \eqref{eq:opt-cond},  and using the condition in \eqref{eq:parameters-cond} that
\begin{align*}
&~\EE\|\vz^{t+1}-\vz^*\|^2\\
\le &\textstyle~\rho_1\|\vx^1-\vx^*\|_{\vD_1-\mu\vI}^2+\|\vz^*\|^2+\sum_{k=1}^t 2\alpha_k\rho_k\left(\sigma^2 + 2\beta ^2 F^2 G^2 + \frac{2G^2}{M}\EE\|\vz^k\|^2\right).
\end{align*}
From the Young's inequality, it follows that $\|\vz^{t+1}\|^2 \le 2\|\vz^{t+1}-\vz^*\|^2 + 2\|\vz^*\|^2$, which together with the above inequality gives the desired result.

\subsection{Proof of Proposition \ref{prop:fix-iter}}
Let $\alpha_k=\frac{\alpha}{K},\forall 1\le k\le K$. It is easy to see that the parameters given in \eqref{eq:para-fix-iter} satisfy the conditions in Proposition \ref{prop:bd-z-prelim}. Hence, for any $t\le K$, it follows from \eqref{eq:bd-z-unified} that
\begin{align}\label{eq:bd-z-1}
\EE\|\vz^{t+1}\|^2
\le \frac{2\rho}{\alpha}\|\vx^{1}-\vx^*\|^2+4\|\vz^*\|^2+4\alpha\rho\big(\sigma^2+2\beta^2 F^2G^2\big)+\frac{8\alpha\rho G^2}{MK}\sum_{k=1}^t\EE\|\vz^k\|^2.
\end{align}
Now we show the result in \eqref{eq:bd-z} by induction. Since $\vz^1=\vzero$, \eqref{eq:bd-z} holds trivially for $k=1$. Assume it holds for $k\le t$. Then from \eqref{eq:bd-z-1}, it follows that
$$
\EE\|\vz^{t+1}\|^2
\le C_1+\frac{8\alpha\rho G^2}{MK}\sum_{k=1}^t\frac{C_1}{1-\frac{8\alpha\rho G^2}{M}} \le \frac{C_1}{1-\frac{8\alpha\rho G^2}{M}},
$$
which completes the proof. 

\subsection{Proof of Proposition \ref{prop:nfix-iter}}

Let $\alpha_k=\frac{\alpha}{\sqrt{k+1}\log (k+1)},\forall k\ge 1$. It is easy to see that the parameters given in \eqref{eq:para-nonfix-iter} satisfy the conditions in Proposition \ref{prop:bd-z-prelim}. Hence, plugging the specified parameters into \eqref{eq:bd-z-unified} gives
\begin{align}\label{eq:1iter-bd-nfix}
\EE\|\vz^{t+1}\|^2\le &~ \textstyle\frac{2\rho}{\alpha}\|\vx^1-\vx^*\|^2+4\|\vz^*\|^2\nonumber\\
&~\textstyle+\sum_{k=1}^t \frac{4\alpha\rho}{(k+1)(\log (k+1))^2}\left(\sigma^2+2\beta^2 F^2 G^2+\frac{2G^2}{M}\EE\|\vz^k\|^2\right).
\end{align}
By 
\begin{align}\label{eq:est-sum-logk}
\sum_{k=1}^\infty \frac{1}{(k+1)(\log (k+1))^2} \le & ~\frac{1}{2(\log 2)^2}+\int_1^{\infty}\frac{1}{(x+1)(\log(x+1))^2}dx \nonumber\\
= &~ \frac{1}{2(\log 2)^2} + \frac{1}{\log 2} \le 2.5, 
\end{align}
we have from \eqref{eq:1iter-bd-nfix} that
\begin{align}\label{eq:1iter-bd-nfix-2}
\EE\|\vz^{t+1}\|^2\le &~ \frac{2\rho}{\alpha}\|\vx^1-\vx^*\|^2+4\|\vz^*\|^2+10\alpha\rho\big(\sigma^2+2 \beta^2 F^2 G^2\big)\nonumber\\ 
&~+ \sum_{k=1}^t \frac{8\alpha\rho}{(k+1)(\log (k+1))^2} \frac{G^2}{M}\EE\|\vz^k\|^2.
\end{align}
Now we show the result in \eqref{eq:bd-z-nfix} by induction. When $k=1$, it obviously holds. Assume the result holds for $k\le t$. Then from \eqref{eq:1iter-bd-nfix-2}, it follows that
\begin{align*}
\EE\|\vz^{t+1}\|^2
\le &~ C_2+\sum_{k=1}^t \frac{8\alpha\rho}{(k+1)(\log (k+1))^2} \frac{G^2}{M}\frac{C_2}{1-\frac{20\alpha\rho G^2}{M}}\\
\le &~  C_2+\frac{20\alpha\rho G^2}{M} \frac{C_2}{1-\frac{20\alpha\rho G^2}{M}} = \frac{C_2}{1-\frac{20\alpha\rho G^2}{M}},
\end{align*}
where the second inequality uses \eqref{eq:est-sum-logk}. This completes the proof.

\subsection{Proof of Proposition \ref{prop:str-cvx}} 

Let $\alpha_k = \frac{\alpha}{k+1},\,\forall\, k\ge 1$. If $\alpha \ge\frac{1}{\mu}$, then $\frac{k+1}{\alpha}\ge\frac{k+2}{\alpha}-\mu$, i.e., $\frac{1}{\alpha_k}\ge \frac{1}{\alpha_{k+1}}-\mu$. Hence, the parameters given in \eqref{eq:para-str-cvx} satisfy the condition in Proposition \ref{prop:bd-z-prelim}, thus \eqref{eq:bd-z-unified} holds and, with the specified parameters, becomes  
\begin{align}\label{eq:bd-z-str-cvx-ineq1}
\EE\|\vz^{t+1}\|^2
\le &~\textstyle\frac{2\rho}{\log(K+1)}\left(\frac{2}{\alpha}-\mu\right)\|\vx^1-\vx^*\|^2+4\|\vz^*\|^2\\
&~~~~+\textstyle\sum_{k=1}^t \frac{4\alpha\rho}{(k+1)\log(K+1)}\left(\sigma^2+2\beta^2F^2G^2+\frac{2G^2}{M}\EE\|\vz^k\|^2\right).\nonumber
\end{align}
Note that for any $t\le K$,
\begin{equation}\label{eq:est-k-logk}
\textstyle\sum_{k=1}^t\frac{1}{k+1}\le \int_1^{t+1}\frac{1}{x}dx=\log (t+1)\le \log(K+1).
\end{equation}
Hence, \eqref{eq:bd-z-str-cvx-ineq1} implies 
\begin{align}\label{eq:bd-z-str-cvx-ineq2}
\EE\|\vz^{t+1}\|^2
\le C_3+\sum_{k=1}^t \frac{4\alpha\rho}{(k+1)\log(K+1)}\frac{2G^2}{M}\EE\|\vz^k\|^2.
\end{align}
Now we show \eqref{eq:bd-z-str-cvx} by induction. When $k=1$, it obviously holds since $\vz^1=\vzero$. Assume \eqref{eq:bd-z-str-cvx} holds for any $k\le t\le K$. Then, from \eqref{eq:est-k-logk} and \eqref{eq:bd-z-str-cvx-ineq2}, it follows that
$$\EE\|\vz^{t+1}\|^2\le C_3+\frac{8\alpha\rho G^2}{M}\frac{C_3}{1-\frac{8\alpha\rho G^2}{M}}= \frac{C_3}{1-\frac{8\alpha\rho G^2}{M}},$$
which completes the proof.

\subsection{Proof of Proposition \ref{prop:adp}}
We first prove \eqref{eq:bd-sum_Dk-3}. Since $\vD_k=\diag(\vs^k)+\frac{\vI}{\alpha_k}$ and $\alpha_k = \frac{\alpha}{\sqrt K},\forall k$, we have for any $1\le t\le K$ that
\begin{align}\label{eq:bd-sum_Dk}
&~\textstyle\sum_{k=1}^t\big(\|\vx^k-\vx\|_{\vD_k}^2-\|\vx^{k+1}-\vx\|_{\vD_k}^2\big)\cr
\le&\textstyle~\|\vx^1-\vx\|_{\vD_1}^2 + \sum_{k=1}^{t-1}\big\langle \vx^{k+1}-\vx, (\vs^{k+1}-\vs^k)\odot (\vx^{k+1}-\vx)\big\rangle\cr
\le &\textstyle~ \frac{1}{\alpha_1}\|\vx^1-\vx\|^2 + B^2 \|\vs^1\|_1 + B^2\sum_{k=1}^{t-1} \|\vs^{k+1}-\vs^k\|_1\cr
=&\textstyle~\frac{1}{\alpha_1}\|\vx^1-\vx\|^2 +B^2 \|\vs^t\|_1,
\end{align}
where $B=\max_{\vx_1,\vx_2\in X}\|\vx_1-\vx_2\|_\infty$, and we have used the fact $s_i^k\ge 0$ and $s_i^{k+1}\ge s_i^k$ for all $i$ and $k$ to have the last equality. By the Cauchy-Schwarz inequality $\langle \vs^t, \vone\rangle \le \|\vone\|\cdot\|\vs^t\| = \sqrt{n}\|\vs^t\|$ and also noting $\|\vs^t\|\le \eta\sqrt{t}$ due to the scaling in Setting \ref{set:adp}, we have from \eqref{eq:bd-sum_Dk} that
\begin{equation}\label{eq:bd-sum_Dk-2}
\textstyle\sum_{k=1}^t\big(\|\vx^k-\vx\|_{\vD_k}^2-\|\vx^{k+1}-\vx\|_{\vD_k}^2\big)\le \frac{1}{\alpha_1}\|\vx^1-\vx\|^2 + \eta B^2\sqrt{nt}.
\end{equation}
Hence, \eqref{eq:bd-sum_Dk-3} holds.

Now let $(\vx,\vz)=(\vx^*,\vz^*)$ in \eqref{eq:1iter-bd} and sum it up from $k=1$ through $t\le K$. Note that the last three expectation terms in \eqref{eq:1iter-bd} vanish when $(\vx,\vz)=(\vx^*,\vz^*)$. Then by \eqref{eq:opt-cond} and Lemma \ref{lem:psi-neg}, and also since $\beta\ge \rho_k=\frac{\rho}{\sqrt{K}},\,\forall k$, we have
\begin{align*}
&~\textstyle\frac{1}{2}\sum_{k=1}^t \EE\|\vx^{k+1}-\vx^*\|_{\vD_k}^2 + \frac{\sqrt K}{2\rho}\EE\|\vz^{t+1}-\vz^*\|^2\cr
\le &~\textstyle \frac{1}{2}\sum_{k=1}^t \EE\|\vx^{k}-\vx^*\|_{\vD_k}^2 + \frac{\sqrt K}{2\rho}\|\vz^{1}-\vz^*\|^2 + \sum_{k=1}^t \alpha_k\left(\sigma^2 + 2\beta ^2 F^2 G^2 + \frac{2G^2}{M}\EE\|\vz^k\|^2\right),
\end{align*}
which together with \eqref{eq:bd-sum_Dk-2} by letting $\vx=\vx^*$ implies
\begin{align*}
\textstyle\frac{\sqrt K}{2\rho}\EE\|\vz^{t+1}-\vz^*\|^2 \le&~ \frac{1}{2\alpha_1}\|\vx^1-\vx^*\|^2 + \frac{\eta B^2\sqrt{nt}}{2}+ \frac{\sqrt K}{2\rho}\|\vz^{1}-\vz^*\|^2\\
&~\textstyle + \sum_{k=1}^t \alpha_k\left(\sigma^2 + 2\beta ^2 F^2 G^2 + \frac{2G^2}{M}\EE\|\vz^k\|^2\right).
\end{align*}
Since $\alpha_k = \frac{\alpha}{\sqrt K},\forall k$ and $\vz^1=\vzero$, multiplying $\frac{2\rho}{\sqrt K}$ to the above inequality and noting $t\le K$ gives
\begin{align*}
\textstyle\EE\|\vz^{t+1}-\vz^*\|^2 \le&~ \frac{\rho}{\alpha}\|\vx^1-\vx^*\|^2 + \rho \eta B^2\sqrt{n}+\|\vz^*\|^2\\
&~ \textstyle+ 2\alpha\rho\left(\sigma^2 + 2\beta ^2 F^2 G^2 + \frac{2G^2}{MK} \sum_{k=1}^t\EE\|\vz^k\|^2\right).
\end{align*}
Hence, by the Young's inequality $\|\vz^{t+1}\|^2 \le 2\|\vz^{t+1}-\vz^*\|^2 + 2\|\vz^*\|^2$, we have from the above inequality that
\begin{align*}
\textstyle\EE\|\vz^{t+1}\|^2 \le &~ \frac{2\rho}{\alpha}\|\vx^1-\vx^*\|^2 + 2\rho \eta B^2\sqrt{n}+4\|\vz^*\|^2\\
&~ \textstyle+ 4\alpha\rho\left(\sigma^2 + 2\beta ^2 F^2 G^2 + \frac{2G^2}{MK} \sum_{k=1}^t\EE\|\vz^k\|^2\right).
\end{align*}
Then following the same arguments as those in the end of the proof of Proposition \ref{eq:para-fix-iter}, we can show the results in \eqref{eq:bd-z-adp}.

\section{Proofs of a few lemmas}

\subsection{Proof of Lemma \ref{lem:z-term}}
Note $\nabla_{z_j}\psi_\beta(f_j(\vx),z_j) = \max\left(-\frac{z_{j}}{\beta}, f_{j}(\vx)\right)$. Then the update of $\vz$ can be written in the compact form
\begin{equation}\label{eq:update-z-compact}
\vz^{k+1}=\vz^k + M\rho_k \ve_{j_k}\odot \nabla_\vz \Psi(\vx^k,\vz^k),
\end{equation}
where $\odot$ denotes componentwise product. Hence,
\begin{align}\label{eq:cross-z}
\textstyle\frac{1}{\rho_k}\langle \vz^k-\vz, \vz^{k+1}-\vz^k\rangle= & ~ \big\langle\vz^k-\vz, \nabla_\vz \Psi(\vx^k,\vz^k) \big\rangle \\
&~\textstyle+ \left\langle\vz^k-\vz, M\ve_{j_k}\odot \nabla_\vz \Psi(\vx^k,\vz^k) - \nabla_\vz \Psi(\vx^k,\vz^k) \right\rangle.\nonumber
\end{align}
Let 
$$J_+^k=\big\{j\in [M]: \beta  f_j(\vx^{k}) + z_j^k \ge0 \big\},\quad J_-^k = [M]\backslash J_+^k.$$
Note that for $\vz\ge\vzero$ and any $j\in J_-^k$, it holds $z_j\big(f_j(\vx^{k})+\frac{z_j^k}{\beta }\big)\le 0$. Then from the definition of $\Psi_\beta$ in \eqref{eq:def-Psi}, one can directly verify that
\begin{align}\label{eq:sum-I12}
&~\textstyle-\Psi_{\beta }(\vx^{k},\vz^k)+\frac{1}{M}\sum_{j=1}^M z_j f_j(\vx^{k}) + \big\langle\vz^k-\vz, \nabla_\vz \Psi(\vx^k,\vz^k) \big\rangle\cr
=&~\textstyle-\frac{1}{M} \sum_{j\in J_+^k}\frac{\beta }{2}\big[f_j(\vx^{k})\big]^2-\frac{1}{M} \sum_{j\in J_-^k}\left[\frac{(z_j^k)^2}{2\beta }-z_j\big(f_j(\vx^{k})+\frac{z_j^k}{\beta }\big)\right]\cr
\le &~\textstyle-\frac{1}{M} \sum_{j\in J_+^k}\frac{\beta }{2}\big[f_j(\vx^{k})\big]^2-\frac{1}{M} \sum_{j\in J_-^k}\frac{(z_j^k)^2}{2\beta }.
\end{align}
 In addition, note 
$$\textstyle-\frac{1}{M} \sum_{j\in J_+^k}\frac{\beta }{2}\big[f_j(\vx^{k})\big]^2-\frac{1}{M} \sum_{j\in J_-^k}\frac{(z_j^k)^2}{2\beta }=-\frac{\beta }{2\rho_k^2}\EE\left[\|\vz^{k+1}-\vz^k\|^2\,\big|\, \cH^{k}\right].$$
Hence, we have the desired result by adding \eqref{eq:cross-z} to \eqref{eq:sum-I12} and using 
$$\textstyle\langle \vz^k-\vz, \vz^{k+1}-\vz^k\rangle=\frac{1}{2}\left[\|\vz^{k+1}-\vz\|^2-\|\vz^{k}-\vz\|^2-\|\vz^{k+1}-\vz^k\|^2\right].$$

\subsection{Proof of Lemma \ref{lem:bd-psi-term}}
For any $j\in [M]$, we have for some $\tilde\nabla f_j(\vx)\in\partial f_j(\vx)$ that
$$\tilde{\nabla}_\vx \psi_{\beta}(f_j(\vx), z_j)=[\beta f_j(\vx)+z_j]_+ \tilde{\nabla} f_j(\vx).$$
From Assumption \ref{assump:bd}, note that $\|\tilde{\nabla} f_j(\vx)\|\le G$ and $[\beta f_j(\vx)+z_j]_+^2 \le 2\beta^2 F^2 + 2 (z_j)^2$. Hence,
$$\|\tilde{\nabla}_\vx \psi_{\beta}(f_j(\vx), z_j)\|^2 \le [\beta f_j(\vx)+z_j]_+^2 \|\tilde{\nabla} f_j(\vx)\|^2\le 2 G^2(\beta^2 F^2 + (z_j)^2),$$
which implies the desired result.

\subsection{Proof of Lemma \ref{lem:bd-cross}}
First note that $\EE\big[\vg_0^k-\tilde{\nabla}f_0(\vx^k)\, \big|\, \cH_k\big] = \vzero$. Hence, if $\vx$ is deterministic, the result in \eqref{eq:bd-cross-1} trivially holds, and similarly if $(\vx,\vz)$ is deterministic, then the results in \eqref{eq:bd-cross-2} and \eqref{eq:bd-cross-3} hold. Next, we prove the results for the stochastic case.

Let $\tilde\vx^1 = \vx^1$ and $\tilde\vx^{k+1}=\tilde\vx^k + \alpha_k(\vg_0^k-\tilde{\nabla}f_0(\vx^k))$ for $1\le k\le K$. Then $\EE\big[\langle \vx^k-\tilde\vx^k, \vg_0^k-\tilde{\nabla}f_0(\vx^k)\rangle\, \big|\, \cH_k\big] = 0$. Hence,
\begin{equation}\label{eq:app-eq1}
\textstyle-\sum_{k=1}^K\alpha_k \EE\big\langle \vx^k-\vx, \vg_0^k-\tilde{\nabla}f_0(\vx^k)\big\rangle = -\sum_{k=1}^K\alpha_k \EE\big\langle \tilde\vx^k -\vx, \vg_0^k-\tilde{\nabla}f_0(\vx^k)\big\rangle.
\end{equation}
In addition, by the definition of $\{\tilde\vx^k\}$, we have
\begin{align*}
\textstyle-\sum_{k=1}^K\alpha_k \langle \tilde\vx^k -\vx, \vg_0^k-\tilde{\nabla}f_0(\vx^k)\rangle
=&~ \textstyle\sum_{k=1}^K \langle \tilde\vx^k -\vx, \tilde\vx^k -\tilde\vx^{k+1}\rangle\\
&\hspace{-2cm}\textstyle~= \frac{1}{2}\left[\tilde\|\vx^1-\vx\|^2  - \|\tilde\vx^{K+1}-\vx\|^2 + \sum_{k=1}^K\|\tilde\vx^k -\tilde\vx^{k+1}\|^2\right]\\
&\hspace{-2cm}\textstyle~\le \frac{1}{2}\left[\|\vx^1-\vx\|^2 + \sum_{k=1}^K\alpha_k^2\|\vg_0^k-\tilde{\nabla}f_0(\vx^k)\|^2\right],
\end{align*} 
where we have used the fact $\tilde\vx^1 = \vx^1$. Substituting the above inequality into \eqref{eq:app-eq1} gives
\begin{equation}\label{eq:app-eq2}
\textstyle-\sum_{k=1}^K\alpha_k \EE\big[\langle \vx^k-\vx, \vg_0^k-\tilde{\nabla}f_0(\vx^k)\rangle\big] \le \frac{1}{2}\EE\left[\|\vx^1-\vx\|^2 + \sum_{k=1}^K\alpha_k^2\|\vg_0^k-\tilde{\nabla}f_0(\vx^k)\|^2\right].
\end{equation}
By Assumption \ref{assump:bd} and the fact $\EE \|\xi - \EE\xi\|^2 \le \EE\|\xi\|^2$ for any random vector $\xi$, we have $\EE\|\vg_0^k-\tilde{\nabla}f_0(\vx^k)\|^2 \le \sigma^2$, and thus \eqref{eq:app-eq2} implies \eqref{eq:bd-cross-1}.

By essentially the same arguments, we can show \eqref{eq:bd-cross-2} by noting 
$\EE\|\vh^k\|^2\le 2\beta ^2 F^2 G^2 + \frac{2G^2}{M}\EE\|\vz^k\|^2$ from \eqref{eq:bd-hk-vec}, and also we can show \eqref{eq:bd-cross-3} by noting from Assumption \ref{assump:bd} that
$$\textstyle\EE\|M\ve_{j_k}\odot \nabla_\vz \Psi(\vx^k,\vz^k)\|^2 = \EE \left|\max\left(-\frac{z_{j_k}}{\beta}, f_{j_k}(\vx^k)\right)\right|^2 \le F^2.$$

\subsection{Proof of Lemma \ref{lem:bd-cross-str}}
Denote $\Delta_\vz^k = M\ve_{j_k}\odot \nabla_\vz \Psi(\vx^k,\vz^k) - \nabla_\vz \Psi(\vx^k,\vz^k)$. Let $\tilde\vz^1 = \vz^1$ and $\tilde\vz^{k+1}=\tilde\vz^k - \rho_k\Delta_\vz^k$ for all $k\ge 1$. Then $\EE\langle \vz^k-\tilde\vz^k, \Delta_\vz^k \rangle = 0$ for any $k$. Note $\rho_k=\frac{\rho}{\log(K+1)}, \, \forall \, k$. Hence,
\begin{align*}
&\textstyle~\sum_{k=1}^K \EE\big\langle\vz^k-\vz, M\ve_{j_k}\odot \nabla_\vz \Psi(\vx^k,\vz^k) - \nabla_\vz \Psi(\vx^k,\vz^k) \big\rangle\\
=&\textstyle~\frac{\log(K+1)}{\rho}\sum_{k=1}^K \EE\big\langle\tilde\vz^k-\vz, \rho_k \Delta_\vz^k\big\rangle\\
=&\textstyle~\frac{\log(K+1)}{\rho}\sum_{k=1}^K \EE\big\langle\tilde\vz^k-\vz, \tilde\vz^k - \tilde\vz^{k+1}\big\rangle\\
=&\textstyle~\frac{\log(K+1)}{2\rho}\EE\left[\|\tilde\vz^1-\vz\|^2 - \|\tilde\vz^{K+1}-\vz\|^2 + \sum_{k=1}^K\|\tilde\vz^k - \tilde\vz^{k+1}\|^2\right]\\
=&\textstyle~\frac{\log(K+1)}{2\rho}\EE\left[\|\vz^1-\vz\|^2 - \|\tilde\vz^{K+1}-\vz\|^2 + \sum_{k=1}^K \rho_k^2 \|\Delta_\vz^k\|^2\right],
\end{align*}
where we have used $\tilde\vz^1=\vz^1$. Since $\EE\|\Delta_\vz^k\|^2 \le \EE\|M\ve_{j_k}\odot \nabla_\vz \Psi(\vx^k,\vz^k)\|^2 = \frac{1}{\rho_k^2}\EE\|\vz^{k+1}-\vz^k\|^2$, we have \eqref{eq:bd-cross-str-3} from the above inequality.

}

\bibliographystyle{abbrv}

\end{document}